\documentclass[10pt]{amsart}
\usepackage{amscd}
\usepackage{verbatim}
\usepackage{amssymb,amsmath}

\input xy
\xyoption{all}

\newtheorem{theorem}{Theorem}[section]
\newtheorem{lemma}[theorem]{Lemma}
\newtheorem{proposition}[theorem]{Proposition}
\newtheorem{corollary}[theorem]{Corollary}

\theoremstyle{definition}
\newtheorem{definition}[theorem]{Definition}

\theoremstyle{remark}
\newtheorem{remark}[theorem]{Remark}
\newtheorem{example}[theorem]{Example}

\makeatletter\let\c@equation=\c@theorem\makeatother

\numberwithin{equation}{section}

\newcommand{\co}{\colon\thinspace}

\newcommand{\colim}{\operatornamewithlimits{colim}}

\newcommand{\hocolim}{\operatornamewithlimits{hocolim}}
\newcommand{\I}{\mathcal I}

\newcommand{\la}{\leftarrow}

\newcommand{\Map}{\operatorname{Map}}

\newcommand{\Sp}{\textit{Sp}}

\newcommand{\U}{\mathcal U}

\newcommand{\xl}{\xleftarrow}
\newcommand{\xr}{\xrightarrow}

\begin{document}

\title{Thom spectra that are symmetric spectra}
\author{Christian Schlichtkrull}
\address{Department of Mathematics, University of Bergen, Johannes
  Brunsgate 12, 5008 Bergen, Norway}
\email{krull@math.uib.no}
\date{\today}

\begin{abstract}
We analyze the functorial and multiplicative properties of the Thom spectrum functor in the setting of symmetric spectra and we establish the relevant homotopy invariance. 
\end{abstract}

\maketitle

\section{Introduction}
The purpose of this paper is to develop the theory of Thom spectra in the setting of symmetric spectra. In particular, we establish the relevant homotopy invariance and we investigate the multiplicative properties. Classically, given a sequence of spaces
$
X_0\to X_1\to X_2\to\dots, 
$
equipped with a compatible sequence of maps $f_n\co X_n\to BO(n)$, the Thom spectrum $T(f)$ is defined by pulling the universal bundles $V(n)$ over $BO(n)$ back via the $f_n$'s and letting 
$$
T(f)_n=\overline{f^*V(n)}/X_n.
$$
Here the bar denotes fibre-wise one-point compactification. More generally, one may  consider compatible families of maps $X_n\to BF(n)$, where $F(n)$ is the topological monoid of base point preserving self-homotopy equivalences of $S^n$, and similarly define a Thom spectrum by pulling back the canonical $S^n$-(quasi)fibration over $BF(n)$. Composing with the canonical maps $BO(n)\to BF(n)$, one sees that the latter construction generalizes the former. This generalization was suggested by Mahowald \cite{Mah1}, \cite{Mah}, and has been investigated in detail by Lewis in \cite{LMS}. 

\subsection{Symmetric Thom spectra via $\I$-spaces}\label{symmetricintro}
In order to translate the definition of Thom spectra into the setting of symmetric spectra, we shall modify the construction by considering certain diagrams of spaces. Let $\I$ be the category whose objects are the finite sets $\mathbf n=\{1,\dots,n\}$, together with the empty set 
$\mathbf 0$, and whose morphisms are the injective maps. The concatenation $\mathbf m\sqcup \mathbf n$ in $\I$ is defined by letting  $\mathbf m$ correspond to the first $m$ and
$\mathbf n$ to the last $n$ elements of $\{1,\dots,m+n\}$. This makes  $\mathcal I$ a symmetric monoidal category with symmetric structure  given by the $(m,n)$-shuffles $\tau_{m,n}\co\mathbf m\sqcup\mathbf n\to\mathbf n\sqcup \mathbf m$. We define an $\I$-space to be a functor from $\I$ to the category $\mathcal U$ of spaces and write $\I\mathcal U$ for the category of such functors. The correspondence $\mathbf n\to BF(n)$ defines an $\I$-space and we show that if 
$X\to BF$ is a map of $\I$-spaces, then the Thom spectrum $T(f)$ defined as above is a symmetric spectrum. The main advantage of the category $\Sp^{\Sigma}$ of symmetric spectra to ordinary spectra is that it has a symmetric monoidal smash product. 
Similarly, the category $\I\mathcal U/BF$ of $\I$-spaces over $BF$ inherits a symmetric monoidal structure from $\I$ and the Thom spectrum functor is compatible with these structures in the following sense.
\begin{theorem}\label{monoidaltheorem}
The symmetric Thom spectrum functor 
$
T\co \I\mathcal U/BF\to \Sp^{\Sigma}
$ 
is strong symmetric monoidal.
\end{theorem}
That $T$ is strong symmetric monoidal means of course that there is a  natural isomorphism of symmetric spectra
$
T(f)\wedge T(g)\cong T(f\boxtimes g),
$
where $f\boxtimes g$ denotes the monoidal product in $\I\mathcal U/BF$. In particular,  $T$ takes monoids in $\I\mathcal U/BF$ to symmetric ring spectra.  A similar construction can be carried out in the setting of orthogonal spectra and the idea of realizing Thom spectra as ``structured ring spectra'' by such a diagrammatic  approach goes back to
\cite{Ma3}.  

\subsection{Lifting space level data to $\I$-spaces}
\label{liftingintro}
Let $\mathcal N$ be the ordered set of non-negative integers, thought of as a subcategory of $\I$ via the canonical subset inclusions. Another starting point for the construction of Thom spectra is 
to consider maps $X\to BF_{\mathcal N}$, where $BF_{\mathcal N}$ denotes the colimit of the 
$\I$-space $BF$ restricted to $\mathcal N$. Given such a map, one may choose a suitable filtration of $X$ so as to get a map of $\mathcal N$-spaces $X(n)\to BF(n)$ and the definition of the Thom spectrum $T(f)$ then proceeds as above. This is the point of view taken by Lewis \cite{LMS}. The space $BF_{\mathcal N}$ has an action of the linear isometries operad $\mathcal L$, and Lewis proves that if $f$ is a map of $\mathcal C$-spaces where $\mathcal C$ is an operad that is augmented over 
$\mathcal L$, then the Thom spectrum $T(f)$ inherits an action of 
$\mathcal C$. 

In the setting of symmetric spectra the problem is how to lift space level data to objects in $\I\mathcal U/BF$. We think of $\I$ as some kind of algebraic structure acting on $BF$, and in order to pull such an action back via a space level map we should ideally map into the quotient space $BF_{\I}$, that is, into the colimit over $\I$. The problem with this approach is that the homotopy type of $BF_{\I}$ differs from that of $BF_{\mathcal N}$. For this reason we shall instead work with the homotopy colimit $BF_{h\I}$ which does have the correct homotopy type. 
We prove in Section \ref{rectificationsection} that the homotopy colimit functor has a right adjoint $U\co \mathcal U/BF_{h\I}\to \I\mathcal U/BF$ such that this pair of adjoint functors defines a Quillen equivalence
\[
\hocolim_{\I}\co
\xymatrix{
\I\mathcal U/BF \ar@<0.5ex>[r] &
\mathcal U/BF_{h\I} \ar@<0.5ex>[l]\thinspace\thinspace\co\!\! U.  
}
\]
Here the model structure on $\I\mathcal U$ is the one established by Sagave-Schlichtkrull
\cite{SS}. The weak equivalences in this model structure are called \emph{$\I$-equivalences} and are the maps that induce weak homotopy equivalences on the associated homotopy colimits; see Section \ref{rightadjointsection} for details. It follows from the theorem that the homotopy theory associated to $\I\mathcal U/BF$ is equivalent to that of $\U/BF_{h\I}$.
As is often the case for functors that are right adjoints, $U$ is only homotopically well-behaved when applied to fibrant objects. We shall usually remedy this by composing with a suitable fibrant replacement functor on $\mathcal U/BF_{h\I}$ and we write $U'$ for the composite functor so defined. 

Composing the right adjoint $U$ with the symmetric Thom spectrum functor from Theorem \ref{monoidaltheorem} we get a Thom spectrum functor on $\mathcal U/BF_{h\I}$. However, even when restricted to fibrant objects this functor does not have all the properties one may expect from a Thom spectrum functor. 
Notably, one of the important properties of the Lewis-May Thom spectrum functor on $\mathcal U/BF_{\mathcal N}$ is that it preserves colimits whereas the symmetric Thom spectrum obtained by composing with $U$ does not have this property. For this reason we shall 
introduce another procedure for lifting space level data to $\I$-spaces in the form of a functor 
\[
R\co \mathcal U/BF_{h\I}\to \I\mathcal U/BF
\] 
and we shall use this functor to associate Thom spectra to objects in $\U/BF_{h\I}$.  The first statement in the following theorem ensures that the functor so defined produces Thom spectra with the correct homotopy type. 

\begin{theorem}\label{RThomfunctor}
There is a natural level equivalence $R\xr{\sim}U'$ over $BF$ and the symmetric Thom spectrum functor defined by the composition
\[
T\co \U/BF_{h\I}\xr{R}\I\U/BF\xr{T} \Sp^{\Sigma}
\]
preserves colimits.
\end{theorem}
As indicated in the theorem we shall use the notation $T$ both for the symmetric Thom spectrum  
on $\I\U/BF$ and for its composition with $R$; the context will always make the meaning clear. In Section \ref{May-Lewis} we show that in a precise sense our Thom spectrum functor becomes equivalent to that of Lewis-May when composing with the forgetful functor from symmetric spectra to spectra. 
We also have the following analogue of Lewis' result imposing $\mathcal L$-actions on Thom spectra. In our setting the relevant operad is the Barrat-Eccles operad 
$\mathcal E$, see \cite{BE} and  \cite{Ma}, Remarks 6.5. We recall that $\mathcal E$ is an $E_{\infty}$ operad and that a space with an $\mathcal E$-action is automatically an associative monoid. 
\begin{theorem}\label{introoperad}
The operad $\mathcal E$ acts on $BF_{h\I}$ and if $f\co X\to BF_{h\I}$ is a map of $\mathcal C$-spaces where $\mathcal C$ is an operad that is augmented over $\mathcal E$, then $T(f)$ inherits an action of $\mathcal C$.
\end{theorem}

We often find that the enriched functoriality obtained by working with homotopy colimits over $\I$ instead of colimits over $\mathcal N$ is very useful. For example, one may represent complexification followed by realification as maps of $\mathcal E$-spaces
$$
BO_{h\I}\to BU_{h\I}\to BO_{h\I},
$$ 
such that the composite $E_{\infty}$ map represents multiplication by $2$.
The procedure for lifting space level data described above works quite generally for diagram categories. Implemented in the framework of orthogonal spectra, it gives an answer to the problem left open in \cite{May4}, Chapter 23, on how to construct orthogonal Thom spectra from space level data; we spell out the details of this in Section \ref{orthogonalsection}. We also remark that one can define an $\I$-space $BGL_1(A)$ for any symmetric ring spectrum $A$, and that an analogous lifting procedure allows one to associate $A$-module Thom spectra to maps $X\to BGL_1(A)_{h\I}$. We hope to return to this in a future paper. 

\subsection{Homotopy invariance}
Ideally, one would like the symmetric Thom spectrum functor to take $\I$-equivalences of $\I$-spaces over $BF$ to stable equivalences of symmetric spectra. However, due to the fact that quasifibrations are not in general preserved under pullbacks this is not true without further assumptions on the objects in $\I\U/BF$. We say that an object $(X,f)$ (that is, a map 
$f\co X\to BF$) is \emph{$T$-good} if $T(f)$ has the same homotopy type as the Thom spectrum associated to a fibrant replacement of $f$; see Definition \ref{T-good-definition} for details. 

\begin{theorem}\label{introinvariance}
If $(X,f)\to(Y,g)$ is an $\I$-equivalence of $T$-good $\I$-spaces over $BF$, then the induced map $T(f)\to T(g)$ is a stable equivalence of symmetric spectra.
\end{theorem}
Here stable equivalence refers to the stable model structure on $\Sp^{\Sigma}$ defined in \cite{HSS} and \cite{MMSS}. It is a subtle property of this model structure that a stable equivalence needs not induce an isomorphism of stable homotopy groups. However, if $X$ and $Y$ are convergent (see  Section \ref{preliminariessection}), then the associated Thom spectra are also convergent, and in this case a stable equivalence is indeed a $\pi_*$-isomorphism in the usual sense. The $T$-goodness requirement in the theorem is not a real restriction since in general any object in $\mathcal U/BF$ can (and should) be replaced by one that is $T$-good. The functor $R$ takes values in the subcategory of convergent $T$-good objects and takes weak homotopy equivalences to $\I$-equivalences (in fact to level-wise equivalences). It follows that the Thom spectrum functor in Theorem \ref{RThomfunctor} is a homotopy functor; see Corollary \ref{TRcorollary}.
  
\begin{example}  
Theorem \ref{introinvariance} also has interesting consequences for Thom spectra that are not convergent. As an example, consider the Thom spectrum $MO(1)^{\wedge\infty}$ that represents the bordism theory of manifolds whose stable normal bundle splits as a sum of line bundles, see \cite{AB}, \cite{Bu}. This is the symmetric Thom spectrum associated to the map of $\I$-spaces $X(n)\to BF(n)$, where $X(n)=BO(1)^n$. 
It is proved in \cite{Sch2} that $X_{h\I}$ is homotopy equivalent to $Q(\mathbb RP^{\infty})$, hence it follows that 
$MO(1)^{\wedge\infty}$ is stably equivalent as a symmetric ring spectrum to the Thom spectrum associated to the map of infinite loop spaces $Q(\mathbb RP^{\infty})\to BF_{h\I}$.
\end{example}
In general any $\I$-space $X$ is $\I$-equivalent to the constant $\I$-space $X_{h\I}$ and consequently any symmetric Thom spectrum is stably equivalent to one arising from a space-level map. However, the added flexibility obtained by working in 
$\I\mathcal U$ is often very convenient. Notably, it is proved in \cite{SS} that any 
$E_{\infty}$ monoid in $\I\mathcal U$ is equivalent to a strictly commutative monoid; something which is well-known not to be the case in $\U$.

\subsection{Applications to the Thom isomorphism}
\label{introThomiso}
As an application of the techniques developed in this paper we present a strictly multiplicative version of the Thom isomorphism. A map $f\co X\to BF_{h\I}$ gives rise to a morphism in 
$\mathcal U/BF_{h\I}$,
$$
\Delta\co (X,f)\to (X\times X,f\circ \pi_2),
$$
where $\Delta$ is the diagonal inclusion and $\pi_2$ denotes the projection onto the second factor of $X\times X$. The Thom spectrum $T(f\circ\pi_2)$ is isomorphic to $X_+\wedge T(f)$, and the Thom diagonal 
$$
\Delta\co T(f)\to X_+\wedge T(f)
$$
is the map of Thom spectra induced by $\Delta$. In Section \ref{orientationsection} we define a canonical orientation $T(f)\to H$, where $H$ denotes (a convenient model of) the Eilenberg-Mac Lane spectrum $H\mathbb Z/2$. Using this we get a map of symmetric spectra
\begin{equation}\label{Z/2Thomequivalence}
T(f)\wedge H\xr{\Delta\wedge H} X_+\wedge T(f)\wedge H\to X_+\wedge H\wedge H\to X_+\wedge H,
\end{equation}
where the last map is induced by the multiplication in $H$. The spectrum level version of the 
$\mathbb Z/2$-Thom isomorphism theorem is the statement that this is a stable equivalence, see \cite{MR}. If $f$ is oriented in the sense that it lifts to a map $f\co X\to BSF_{h\I}$, then we define a canonical integral orientation $T(f)\to H\mathbb Z$ and the spectrum level version of the integral Thom isomorphism theorem is the statement that the induced map
\begin{equation}\label{ZThomequivalence}
T(f)\wedge H\mathbb Z\to X_+\wedge H\mathbb Z
\end{equation}
is a stable equivalence. In our framework these results lift to ``structured ring spectra'' in the sense of the following theorem. Here $\mathcal C$ again denotes an operad that is augmented over $\mathcal E$.

\begin{theorem}\label{structuredThomiso} 
If $f\co X\to BF_{h\I}$ (respectively $f\co X\to BSF_{h\I}$) is a map of $\mathcal C$-spaces, then the spectrum level Thom equivalence (\ref{Z/2Thomequivalence}) (respectively 
(\ref{ZThomequivalence})) is a $\mathcal C$-map.
\end{theorem}

For example, one may represent the complex cobordism spectrum $MU$ as the Thom spectrum associated to the $\mathcal E$-map $BU_{h\I}\to BSF_{h\I}$ and the Thom equivalence 
(\ref{ZThomequivalence}) is then an equivalence of $E_{\infty}$ symmetric ring spectra. 
This should be compared with the $H_{\infty}$ version in \cite{LMS}. 

\subsection{Diagram Thom spectra and symmetrization}
The definition of the symmetric Thom spectrum functor shows that the category $\I$ is closely related to the category of symmetric spectra. However, many of the Thom spectra that occur in the applications do not naturally arise from a map of $\I$-spaces but rather from a map 
of $\mathcal D$-spaces for some monoidal category $\mathcal D$ equipped with a monoidal functor $\mathcal D\to\I$.  We formalize this in Section  \ref{diagramThomsection} where we introduce the notion of a $\mathcal D$-spectrum associated to such a monoidal functor. For example, the complex cobordism spectrum $MU$ associated to the unitary groups $U(n)$ and the Thom spectrum $M\mathfrak B$ associated to the braid groups $\mathfrak B(n)$ can be realized as diagram ring spectra in this way. 
It is often convenient to replace the $\mathcal D$-Thom spectrum associated to a map of 
$\mathcal D$-spaces $f\co X\to BF$ by a symmetric spectrum, and our preferred way of doing this is to first transform $f$ to a map of $\mathcal I$-spaces and then evaluate the symmetric Thom spectrum functor on this transformed map. In this way we end up with a symmetric spectrum to which we can exploit the structural relationship to the category of 
$\I$-spaces. We shall discuss various ways of carrying out this ``symmetrization'' process and in particular we shall see how to realize the Thom spectra $MU$ and $M\mathfrak B$ as (in the case of $MU$ commutative) symmetric ring spectra. 

\subsection{Organization of the paper}
We begin by recalling the basic facts about Thom spaces and Thom spectra in Section \ref{preliminariessection}, and in Section \ref{symmetricsection} we introduce the symmetric Thom spectrum functor and show that it is strong symmetric monoidal. The $\I$-space lifting functor $R$ is introduced in Section \ref{rectificationsection}, where we prove Theorem \ref{RThomfunctor} in a more precise form; this is the content of Proposition \ref{RU'comparison} and Corollary \ref{TRcorollary}. Here we also compare the Lewis-May Thom spectrum functor to our construction. We prove the homotopy invariance result Theorem \ref{introinvariance} in Section \ref{invariancesection}, and in Section \ref{Thomoperadsection} we analyze to what extent the constructions introduced in the previous sections are preserved under operad actions. In particular, we prove Theorem 
\ref{introoperad} in a more precise form; this is the content of Corollary \ref{TUoperad}. The Thom isomorphism theorem is proved in Section \ref{Thomisosection} and in Section \ref{diagramThomsection} we discuss how to symmetrize other types of diagram Thom spectra and how the analogue of the lifting functor $R$ works in the context of orthogonal spectra. Finally, we have included some background material on homotopy colimits in Appendix \ref{hocolimsection}.

\subsection{Notation and conventions}
We shall work in the categories $\mathcal U$ and $\mathcal T$  
of unbased and based compactly generated weak Hausdorff spaces. By a cofibration we understand a map having the homotopy extension property, see \cite{St}. 
A based space is well-based if the inclusion of the base point is a cofibration. 
In this paper $S^n$ always denotes the one-point compactification of $\mathbb R^n$.
By a spectrum $E$ we understand a sequence $\{E_n\co n\geq 0\}$ of
based spaces together with a sequence of based structure 
maps $S^1\wedge E_n\to E_{n+1}$. 
A map of spectra $f\co E\to F$ is a sequence of based maps $f_n\co E_n\to F_n$
that commute with the structure maps and we write $\Sp$ for the category of
spectra so defined. A spectrum is \emph{connective} if $\pi_n(E)=0$ for $n<0$ and
\emph{convergent} if there is an unbounded, non-decreasing sequence of integers 
$\{\lambda_n:n\geq 0\}$ such that the adjoint structure maps 
$E_n\to \Omega E_{n+1}$ are $(\lambda_n+n)$-connected for all $n$.

\section{Preliminaries on Thom spaces and Thom spectra}\label{preliminariessection}
In this section we recall the basic facts about Thom spaces and Thom spectra that we shall need.  The main reference for this material is Lewis' account in \cite{LMS}, Section IX. Here we emphasize the details relevant for the construction of symmetric Thom spectra in Section 
\ref{symmetricsection}. We begin by recalling the two-sided simplicial bar construction and some of its properties, referring to \cite{Ma2} for more details. Given a 
topological monoid $G$, a right $G$-space $Y$, and a left $G$-space
$X$, this is the simplicial space $B_{\bullet}(Y,G,X)$ with $k$-simplices
$Y\times G^k\times X$ and simplicial operators
$$
d_i(y,g_1,\dots,g_k,x)=
\begin{cases}
(yg_1,g_2\dots,g_k,x),&\text{for $i=0$}\\
(y,g_1,\dots,g_ig_{i+1},\dots,g_k,x),&\text{for $0<i<k$}\\
(y,g_1,\dots,g_kx),&\text{for $i=k$},
\end{cases}
$$
and
$$
s_i(y,g_1,\dots,g_k,x)=(y,\dots,g_{i-1},1,g_i,\dots,x),
\quad\text{for $0\leq i\leq k$}.
$$
We write $B(Y,G,X)$ for the topological realization. In the case
where $X$ and $Y$ equal the one-point space $*$, this is the usual
simplicial construction of the classifying space $BG$.
The projection of $X$ onto $*$ induces a map 
$$
p\co B(Y,G,X)\to B(Y,G,*)
$$ 
whose fibres are homeomorphic to $X$. Furthermore, if
$X$ has a $G$-invariant basepoint, then the inclusion of the base point
defines a section 
$$
s:B(Y,G,*)\to B(Y,G,X).
$$ 
Recall that a topological monoid is \emph{grouplike} if 
the set of components with the induced monoid structure is a group. 
\begin{proposition}[\cite{Le},\cite{Ma2}]\label{barprop}
If $G$ is a well-based grouplike monoid, then the projection $p$ is a
quasifibration, and if $X$ has a $G$-invariant base point such that $X$ is  
(non-equivariantly) well-based, then the section $s$ is a cofibration.
\qed
\end{proposition}
In general we say that a sectioned quasifibration is \emph{well-based} if the section is a cofibration. Let $F(n)$ be the topological monoid of base point preserving homotopy
equivalences of $S^n$, where we recall that the latter denotes the
one-point compactification of $\mathbb R^n$. It follows from
\cite{Le}, Theorem 2.1, that this is a well-based monoid and we let
$V(n)=B(*,F(n),S^n)$. Then $BF(n)$ is
a classifying space for sectioned fibrations with fibre equivalent to $S^n$
and the projection $p_n\co V(n)\to BF(n)$ is a well-based quasifibration.  Given a map $f\co X\to
BF(n)$, let $p_X\co f^*V(n)\to X$ be the pull-back of $V(n)$ along $f$,
and notice that the section $s$ gives rise to a section $s_X\co X\to f^*V(n)$.
The associated Thom space is the quotient space
$$
T(f)=f^*V(n)/s_X(X).
$$
This construction is clearly functorial on the category $\mathcal U/BF(n)$ of spaces over 
$BF(n)$. We often use the notation $(X,f)$ for an object $f\co X\to BF(n)$ in this category. 
In order for the Thom space functor to be homotopically well-behaved we would like
$p_X$ to be a quasifibration and $s_X$ to be a cofibration, but
unfortunately this is not true in general. This is the main technical
difference compared to working with sectioned fibrations.  For our purpose 
it will not do to replace the quasifibration $p_n$ by an equivalent
fibration since we then loose the strict multiplicative  properties of the bar construction required for the definition of strict multiplicative structures on Thom spectra.    
We say that $f$ \emph{classifies a well-based quasifibration} 
if  $p_X$ is a quasi-fibration and $s_X$ is a cofibration. The following well-known results are included here for completeness.  
\begin{lemma}[\cite{LMS}]\label{wellbasedfibrationlemma}
Given a well-based sectioned quasifibration $p\co V\to B$ and a Hurewicz fibration $f\co X\to B$, the pullback $p_X\co f^*V\to X$ is again a well-based quasifibration.
\end{lemma}
\begin{proof}
Since $f$ is a fibration the pullback diagram defining $f^*V$ is homotopy cartesian, hence 
$f^*V\to X$ is a quasifibration. In order to see that the section $s_X$ is a cofibration, notice that it is the pullback of the section of $p$ along the Hurewicz fibration $f^*V\to V$. The result then follows from Theorem 12 of  \cite{Str} which states that the pullback of a cofibration along a Hurewicz fibration is again a cofibration. 
\end{proof}

Let $\mathit{Top}(n)$ be the topological group of base
point preserving homeomorphisms of $S^n$. The next result is the main reason why the objects in $\mathcal U/BF(n)$ that factor through $B\mathit {Top}(n)$ are easier to handle than general objects.
\begin{lemma}\label{Top(n)-lemma}
If $f\co X\to BF(n)$ factors through $B\mathit{Top}(n)$, 
then $f$ classifies a well-based Hurewicz  fibration, hence a 
well-based quasifibration. \qed
\end{lemma}
\begin{proof}
Let $W(n)=B(*,\mathit{Top}(n),S^n)$. The projection $W(n)\to
B\mathit{Top}(n)$ is a fibre bundle by \cite{Ma2}, Corollary 8.4, and in particular a Hurewicz
fibration. Suppose that 
$f$ factors through a map $g\co X\to B\mathit{Top}(n)$. Then $f^*V(n)$ is 
homeomorphic to $g^*W(n)$ and thus $p_X$ is a Hurewicz
fibration. We must prove that the section is a cofibration. Let us use the Str\o m model structure \cite{Str2} on $\U$ to get a factorization $g=g_2g_1$ where $g_1$ is a cofibration and $g_2$ is a Hurewicz fibration. From this we get a factorization of the pullback diagram defining $g^*W(n)$,
\[
\xymatrix{
g^*W(n) \ar[r] \ar[d]^{p_X} & g_2^*W(n) \ar[r] \ar[d]^{p_Y}& W(n)\ar[d]\\
X \ar[r]^-{g_1} & Y \ar[r]^-{g_2} & B\mathit{Top}(n), 
}
\]
and it follows from Lemma \ref{wellbasedfibrationlemma} that the section $s_Y$ of $p_Y$ is a cofibration. Since $p_Y$ is a Hurewicz fibration it follows by the same argument
that the induced map $g^*W(n)\to g_2^*W(n)$ is also a cofibration. It is
clear that the composition $X\to g^*W(n)\to g_2^*W(n)$ is a cofibration and
the conclusion thus follows from Lemma 5 of \cite{Str2}, 
which states that if $h=i\circ j$ is a composition of maps in which
$h$ and $i$ are both cofibrations, then $j$ is a cofibration as well. 
\end{proof}

This lemma applies in particular if $f$ factors through $BO(n)$. 
In order to get around the difficulty that the Thom space functor is not a homotopy functor on the whole category $\mathcal U/BF(n)$ we 
follow Lewis \cite{LMS}, Section IX, and define a functor 
\begin{equation}\label{fibrantreplacement}
\Gamma\co \mathcal U/BF(n)\to\mathcal U/BF(n),\quad (X,f)\mapsto (\Gamma_f(X),\Gamma(f))
\end{equation}
by replacing a map by a (Hurewicz) fibration in the usual way,
$$
\Gamma_f(X)=\{(x,\omega)\in X\times BF(n)^I\co f(x)=\omega(0)\} ,\quad\Gamma(f)(x,\omega)
=\omega(1).
$$
We sometimes write $\Gamma(X)$ instead of $\Gamma_f(X)$ when the map $f$ is clear from the context. The natural inclusion $X\to\Gamma_f(X)$, whose second coordinate is the constant path at $f(x)$, defines a natural equivalence from the identity functor on $\mathcal U/BF(n)$ to 
$\Gamma$. It follows from Lemma \ref{wellbasedfibrationlemma} that the composition of the Thom space functor $T$ with $\Gamma$ is a homotopy functor.
We think of $T(\Gamma(f))$ as representing the correct homotopy type of the Thom space and say that $f$ is \emph{$T$-good} if the natural map $T(f)\to T(\Gamma(f))$ is a weak homotopy equivalence. In particular, $f$ is $T$-good if it classifies a well-based quasifibration. It follows from the above discussion that the restriction of $T$ to the subcategory of $T$-good objects is a homotopy functor. 
\begin{remark}
The fibrant replacement functors used in \cite{LMS} and \cite{Ma2} are defined using Moore paths instead of paths defined on the unit interval $I$. The use of Moore paths is less convenient for our purposes since we shall use $\Gamma$ in combination with more general homotopy pullback constructions. 
 \end{remark}

The following basic lemma is needed in order to establish the connectivity and convergence properties of the Thom spectrum functor. It may for example be deduced from the dual Blakers-Massey Theorem in \cite{G}.
\begin{lemma}\label{basicthomlemma}
Let
$$
\begin{CD}
V_1@>>> V_2 \\
@VV p_1 V @VV p_2 V\\
B_1@>\beta>> B_2
\end{CD}
$$
be a pullback diagram of well-based quasifibrations $p_1$ and
$p_2$. Suppose that  $p_1$ and $p_2$ are
$n$-connected with $n>1$ and that $\beta$ is $k$-connected. Then the
quotient spaces $V_1/B_1$ and $V_2/B_2$ are well-based and $(n-1)$-connected, and the 
induced map $V_1/B_1\to V_2/B_2$ is $(k+n)$-connected.\qed
\end{lemma}

We now turn to Thom spectra. Let $\mathcal N$ be as in Section \ref{liftingintro}, and write $\mathcal N\mathcal U$ for the category of $\mathcal N$-spaces, that is, functors 
$X\co \mathcal N\to\mathcal U$. Consider the $\mathcal N$-space
$BF$ defined by the sequence of cofibrations
$$
BF(0)\stackrel{i_0}{\to}BF(1)\stackrel{i_1}{\to}BF(2)\stackrel{i_2}
{\to}\dots   
$$
obtained by applying $B$ to the monoid homomorphisms $F(n)\to F(n+1)$ 
that take an element $u$ to $1_{S^1}\!\wedge u$, the smash product with the identity on $S^1$.  
Notice, that there are pullback diagrams 
\begin{equation}\label{Npullback}
\begin{CD}
S^1\bar\wedge V(n)@>>> V(n+1)\\
@VVV @VVV\\
BF(n)@>i_n>> BF(n+1),
\end{CD}
\end{equation}
where $S^1\bar\wedge-$ denotes fibre-wise smash product with $S^1$. Indeed, there clearly is such a pullback diagram of the underlying simplicial spaces, and topological realization preserves pullback diagrams. We let $\mathcal N\mathcal U/BF$ be the category of $\mathcal
N$-spaces over $BF$. Thus, an object is a sequence of maps
$$
f_n\co X(n)\to BF(n)
$$
that are compatible with the structure maps.
Again we may specify the domain by writing the objects in the 
form $(X,f)$. 
\begin{definition}\label{NThomdefinition}
The Thom spectrum functor $T\co\mathcal N\mathcal U/BF\to \textit{Sp}$ is
defined by applying the Thom space construction level-wise, $T(f)_n=T(f_n)$, 
with structure maps given by
$$
S^1\wedge T(f_n)\cong T(i_n\circ f_n)\to T(f_{n+1}).
$$
\end{definition}
A morphism in $\mathcal N\mathcal U/BF$ induces a map of Thom spectra in the obvious
way. As for the Thom space functor, the Thom spectrum functor is not homotopically well-behaved on the whole category $\mathcal N\mathcal U/BF$. We define a functor
$
\Gamma\co \mathcal N\mathcal U/BF\to \mathcal N\mathcal U/BF
$
by applying the functor $\Gamma$ level-wise, and we say that an object $(X,f)$ is 
\emph{$T$-good} if the induced map $T(f)\to T(\Gamma(f))$ is a stable equivalence. We say that $f$ is \emph{level-wise $T$-good} if the induced map is a level-wise equivalence. The following proposition is an immediate consequence of Lemma \ref{wellbasedfibrationlemma} and Lemma \ref{basicthomlemma}. 
\begin{proposition}
If $f\co X\to BF$ is $T$-good, then $T(f)$ is connective.\qed
\end{proposition}
An $\mathcal N$-space $X$ is said to be \emph{convergent} if there exists
an unbounded, non-decreasing sequence of integers 
$\{\lambda_n\co n\geq0\}$ such that $X(n)\to X(n+1)$ is
$\lambda_n$-connected for each $n$.
\begin{proposition}\label{Tconvergent}
If $f\co X\to BF$ is level-wise $T$-good and $X$ is convergent, then $T(f)$ is also convergent.
\end{proposition}
\begin{proof}
We may assume that $f$ is a level-wise fibration, hence classifies a well-based quasifibration at each level. If $X(n)\to X(n+1)$ is $\lambda_n$-connected, 
it follows from Lemma \ref{basicthomlemma} that the
structure map $S^1\wedge T(f_n)\to T(f_{n+1})$ is
$(\lambda_n+n)$-connected. The convergence of $X$ thus implies
that of $T(f)$.
\end{proof}

Given an $\mathcal N$-space $X$, write $X_{h\mathcal N}$ for its homotopy colimit. This is homotopy equivalent to the usual telescope construction on $X$. We say that a morphism $(X,f)\to (Y,g)$ in $\mathcal N\mathcal U/BF$ is an \emph{$\mathcal N$-equivalence} if the induced map of homotopy colimits $X_{h\mathcal N}\to Y_{h\mathcal N}$ is a weak homotopy equivalence. 
The following theorem can be deduced from \cite{LMS}, Proposition 4.9. We shall indicate a more direct proof in Section \ref{NTinvariancesection}. 
\begin{theorem}\label{NThominvariance}
If $(X,f)\to (Y,g)$ is an $\mathcal N$-equivalence of $T$-good $\mathcal N$-spaces over $BF$, then the induced map $T(f)\to T(g)$ is a stable equivalence.
\end{theorem}
In particular, it follows that $T\circ\Gamma$ takes $\mathcal N$-equivalences to stable equivalences.

\section{Symmetric Thom spectra}\label{symmetricsection}
We begin by recalling the definition of a (topological) symmetric spectrum. The
basic references are the papers \cite{HSS} and \cite{MMSS} that deal
respectively with the simplicial and the topological version of the
theory. See also \cite{Schw}.
\subsection{Symmetric spectra}\label{symmetricspectrumsection}
By definition a symmetric spectrum $X$ is a spectrum in which
each of the spaces $X(n)$ is equipped with a base point preserving left 
$\Sigma_n$-action, such that the iterated structure maps 
$$
\sigma^n\co S^m\wedge X(n)\to X(m+n)
$$
are $\Sigma_m\times \Sigma_n$-equivariant. A map of symmetric spectra
$f\co X\to Y$ is a sequence of $\Sigma_n$-equivariant based maps
$X(n)\to Y(n)$ that strictly commute with the structure maps.  We
write $\Sp^{\Sigma}$ for the category of symmetric spectra. 
Following \cite{MMSS} we
shall view symmetric spectra as diagram spectra, and for this reason we introduce some  notation which will be convenient for our purposes. Let the category $\mathcal I$ be as in Section 
\ref{symmetricintro}. Given a morphism $\alpha\co\mathbf
m\to\mathbf n$ in $\mathcal 
I$, let $\mathbf n-\alpha$ denote the complement of $\alpha(\mathbf
m)$ in $\mathbf n$ and let $S^{n-\alpha}$ be the one-point
compactification of $\mathbb R^{\mathbf n-\alpha}$.
Consider then the topological category
$\I_S$ that has the same objects as $\I$, but whose morphism
spaces are defined by 
$$
\I_S(\mathbf m, \mathbf n)=\bigvee_{\alpha\in\I(\mathbf m,\mathbf
  n)}S^{n-\alpha}.
$$
We view $\I_S$ as a category enriched in the category of based spaces $\mathcal T$. Writing the morphisms in the form 
$(\mathbf x,\alpha)$ for $\mathbf x\in S^{n-\alpha}$, the composition is defined by 
$$
\I_S(\mathbf m,\mathbf n)\wedge 
\I_S(\mathbf l,\mathbf m)\to
\I_S(\mathbf l,\mathbf n),\quad (\mathbf x,\alpha)\wedge(\mathbf
y,\beta)\mapsto(\mathbf x\wedge\alpha_*\mathbf y,\alpha\beta),
$$  
where $\mathbf x\wedge \alpha_*\mathbf y$ is defined by the canonical homeomorphism 
$$
S^{n-\alpha}\wedge S^{m-\beta}\cong S^{n-\alpha\beta},\quad \mathbf x\wedge\mathbf y\mapsto \mathbf x\wedge\alpha_*\mathbf y,
$$
obtained by reindexing the coordinates of $S^{m-\beta}$ via $\alpha$. This choice of notation has the advantage of making some of our constructions self-explanatory. By a functor between categories enriched in $\mathcal T$ we understand a functor such that the maps of morphism spaces are based and continuous. 
Thus, if $X\co\I_S\to\mathcal T$ is a functor in this sense, then we have for each morphism 
$\alpha\co\mathbf m\to\mathbf n$ in $\I$ a based continuous map
$$
\alpha_*\co S^{n-\alpha}\wedge X(m)\to X(n).
$$
One easily checks that the maps $S^1\wedge X(n)\to X(n+1)$ induced by the morphisms 
$\mathbf n\to 1\sqcup \mathbf n$ give $X$ the structure of a symmetric spectrum and that the category of (based continuous) functors $\I_S\to\mathcal T$ may be identified with $\Sp^{\Sigma}$ in this way.   
The symmetric monoidal structure of $\I$ gives rise to a symmetric
monoidal structure on $\I_S$. On morphism spaces this is given by
the continuous maps
$$
\sqcup\co
\I_S(\mathbf m_1,\mathbf n_1)\times
\I_S(\mathbf m_2,\mathbf n_2) \to
\I_S(\mathbf m_1\sqcup\mathbf m_2,\mathbf n_1\sqcup\mathbf n_2),\quad
$$
that map a pair of morphisms 
$\big((\mathbf x,\alpha),(\mathbf y,\beta)\big)$ 
to $(\mathbf x\wedge \mathbf y, \alpha\sqcup \beta)$.
As noted in \cite{MMSS}, this implies that the category of symmetric
spectra inherits a symmetric monoidal smash product. Given
symmetric spectra $X$ and $Y$, this is defined by the 
left Kan extension,
\begin{equation}\label{Kansmash}
X\wedge Y(n)=\colim_{\mathbf n_1\sqcup\mathbf n_2\to\mathbf n}
X(n_1)\wedge Y(n_2),
\end{equation}
where the colimit is over the topological category 
$(\sqcup\downarrow\mathbf n)$ of objects and morphisms in $\I_S\times \I_S$ over $\mathbf n$. More explicitly, we may rewrite this as
$$
X\wedge Y(n)=\colim_{\alpha\co\mathbf n_1\sqcup\mathbf n_2\to\mathbf n}
S^{n-\alpha}\wedge X(n_1)\wedge Y(n_2),
$$ 
where the colimit is now over the discrete category $(\sqcup\downarrow\mathbf n)$ of objects and morphisms in $\I\times \I$ over $\mathbf n$. Given a morphism in this category of the form 
$$
(\beta_1,\beta_2)\co (\mathbf n_1,\mathbf n_2,\mathbf n_1\sqcup\mathbf n_2\xr{\alpha}\mathbf n) 
\to(\mathbf n'_1,\mathbf n'_2,\mathbf n'_1\sqcup\mathbf n'_2\xr{\alpha'}\mathbf n), 
$$
the morphism $\alpha'$ specifies a homeomorphism
$$
S^{n-\alpha}\cong S^{n-\alpha'}\wedge S^{n_1'-\beta_1}\wedge S^{n_2'-\beta_2},
$$
and the induced map in the diagram is defined by
\begin{align*}
S^{n-\alpha}\wedge X(n_1)\wedge Y(n_2)
&\to S^{n-\alpha'}\wedge S^{n_1'-\beta_1}\wedge X(n_1)
\wedge S^{n_2'-\beta_2}\wedge Y(n_2)\\
&\to S^{n-\alpha'}\wedge X(n_1')\wedge Y(n_2').
\end{align*}
The unit for the smash product is the sphere spectrum $S$ with $S(n)=S^n$.  
By definition, a \emph{symmetric ring spectrum} is a monoid in this monoidal
category. Spelling this out using the above notation, a symmetric ring spectrum is a symmetric spectrum $X$ equipped with a based map $1_X\co S^0\to X(0)$ and a collection of based maps 
$$
\mu_{m,n}\co X(m)\wedge X(n)\to X(m+n),
$$ 
such that the usual unitality and associativity conditions hold, and such that the diagrams 
$$
\begin{CD}
S^{n-\alpha}\wedge X(m)\wedge S^{n'-\alpha'}
\wedge X(m') @>\mu_{m,m'}\circ
\textit{tw}>>
S^{n+n'-\alpha\sqcup \alpha'}\wedge X(m+m')\\
@VV \alpha\wedge \alpha' V  @VV \alpha\sqcup\alpha' V
\\
X(n)\wedge X(n') @>\mu_{n,n'}>> X(n+n')
\end{CD}
$$
commute for each pair of morphisms $\alpha\co\mathbf m\to \mathbf n$ and
$\alpha'\co\mathbf m'\to \mathbf n'$ in $\I$. Here $\textit{tw}$ flips the factors $X(m)$ and 
$S^{n'-\alpha'}$.  A ring spectrum is \emph{commutative} if it defines a commutative
monoid in $\Sp^{\Sigma}$. Explicitly, this means that there are commutative diagrams
$$
\begin{CD}
X(m)\wedge X(n)@>\mu_{m,n}>>X(m+n)\\
@VV\textit{tw} V @VV\tau_{m,n} V\\
X(n)\wedge X(m)@>\mu_{n,m}>>X(n+m),
\end{CD}
$$
where the right hand vertical map is given by the left action of the $(m,n)$-shuffle $\tau_{m,n}$.

\subsection{Symmetric Thom spectra via $\I$-spaces}\label{Thomdiagramsection}
As in Section \ref{symmetricintro} we write $\I\mathcal U$ for the category of $\I$-spaces. This category inherits the structure of a symmetric monoidal category from that of $\mathcal I$ in the usual way: given $\mathcal I$-spaces $X$ and
$Y$, their product $X\boxtimes Y$ is defined by the Kan extension
$$
(X\boxtimes Y)(n)=\colim_{\mathbf n_1\sqcup \mathbf n_2\to \mathbf n}  
X(n_1)\times Y(n_2),
$$
where the colimit is again over the category $(\sqcup\downarrow\mathbf n)$. The unit for the monoidal structure is the constant $\I$-space $\I(\mathbf 0,-)$. We use term  \emph{$\I$-space monoid} for a monoid in this category. This amounts to an 
$\mathcal I$-space $X$ equipped with a unit $1_X\in X(0)$ and 
a natural transformation of $\I\times \I$-diagrams
$$
\mu_{m,n}:X(m)\times X(n)\to X(m+n)
$$
that satisfies the obvious associativity and unitality conditions. 
An $\I$-space monoid $X$ is commutative if it defines a commutative monoid in 
$\mathcal I\mathcal U$, that is, the diagrams
$$
\begin{CD}
X(m)\times X(n)@>\mu_{m,n} >> X(m+n)\\
@VV\textit{tw}V @VV\tau_{m,n}V \\
X(n) \times X(m) @> \mu_{n,m} >> X(n+m)
\end{CD}
$$
are commutative. 
   
The family of topological monoids $F(n)$ define a functor from $\I$ to the category of topological monoids: a morphism $\alpha\co\mathbf m\to\mathbf n$ in $\I$
induces a monoid homomorphism $\alpha\co F(m)\to F(n)$ by associating to an element 
$f$ in $F(m)$ the composite map
\begin{equation}\label{Ifunctoriality}
S^n\cong S^{n-\alpha}\wedge S^m\xrightarrow{S^{n-\alpha}\wedge
f}S^{n-\alpha}\wedge S^m\cong S^n.
\end{equation}
As usual the homeomorphism $S^{n-\alpha}\wedge S^m\cong S^n$ is induced
by the bijection $(\mathbf n-\alpha)\sqcup\mathbf m\to \mathbf n$
specified by $\alpha$ and the inclusion of $\mathbf n-\alpha$ in
$\mathbf n$. We also have the natural monoid homomorphisms 
$$
F(m)\times F(n)\to F(m+n),\quad (f,g)\mapsto f\wedge g
$$
defined by the usual smash product of based spaces. Applying the classifying space functor degree-wise and using that it commutes with products, we get from this the commutative  
$\I$-space monoid $BF\co\mathbf n\mapsto BF(n)$. 
We write $\mathcal I\mathcal U/BF$ for the category of $\mathcal I$-spaces over $BF$ with objects $(X,f)$ given by a map $f\co X\to BF$ of
 $\mathcal I$-spaces. This category inherits a
symmetric monoidal structure from that of $\mathcal I\mathcal
U$: given objects $(X,f)$ and $(Y,g)$, the product is defined by
the composition
$$
f\boxtimes g\co X\boxtimes Y\stackrel{f\boxtimes g}{\longrightarrow} 
BF\boxtimes BF\to BF,
$$
where the last map is the multiplication in $BF$. The meaning of the symbol 
$f\boxtimes g$ will always be clear from the context. 
By definition, a monoid in this monoidal structure is a pair $(X,f)$ given by an
$\mathcal I$-space monoid $X$ together with a monoid morphism $f\co X\to BF$. 
\begin{definition}
The symmetric Thom spectrum functor 
$
T\co\mathcal I\mathcal U/BF\to
\Sp^{\Sigma}
$
is defined by the level-wise Thom space construction $T(f)(n)=T(f_n)$. 
A morphism $\alpha\co\mathbf m\to \mathbf n$ in $\mathcal I$ 
gives rise to a pullback diagram
$$
\begin{CD}
S^{n-\alpha}\bar\wedge V(m) @>>> V(n) \\
@VVV @VVV \\
BF(m) @>\alpha >> BF(n)
\end{CD}
$$
in which $\bar\wedge$ denotes the fibre-wise smash product.
On fibres this restricts to the homeomorphism $S^{n-\alpha}\wedge
S^m\to S^n$ specified by $\alpha$. Pulling this diagram back via
$f$ and applying the Thom space construction, we get the required 
structure maps 
$$
S^{n-\alpha}\wedge T(f_m)\cong T(\alpha\circ f_m)\to T(f_n).
$$    
\end{definition}

Notice, that this Thom spectrum functor is related to that in Section \ref{preliminariessection} by a commutative diagram of functors
\begin{equation}\label{INThom}
\begin{CD}
\I\mathcal U/BF@>T>> \Sp^{\Sigma}\\
@VVV @VVV\\
\mathcal N\mathcal U/BF @>T>> \Sp,
\end{CD}
\end{equation}
where the vertical arrows represent the obvious forgetful functors. Recall the notion of a strong symmetric monoidal functor from \cite{Mac}, Section XI.2. We now prove Theorem \ref{monoidaltheorem} stating that the symmetric Thom spectrum is strong symmetric monoidal. 

\begin{proof}[Proof of Theorem \ref{monoidaltheorem}]
It is clear that we have a canonical isomorphism $S\to T(*)$. We must show that given objects 
$(X,f)$ and $(Y,g)$ in $\I\U/BF$ there is a natural isomorphism 
$$
T(f)\wedge T(g)\cong T(f\boxtimes g).
$$
By definition, $X\boxtimes Y(n)$ is the colimit of the $(\sqcup\downarrow\mathbf n)$-diagram
$$
(\mathbf n_1,\mathbf n_2, \mathbf n_1\sqcup\mathbf n_2\xr{\alpha} \mathbf n)\mapsto 
X(n_1)\times Y(n_2).
$$
Given $\alpha\co\mathbf n_1\sqcup\mathbf n_2\to\mathbf n$, let $\alpha(f,g)$ be the composite map
$$
X(n_1)\times Y(n_2)\xr{f_{n_1}\times g_{n_2}} BF(n_1)\times BF(n_2)\to BF(n_1+n_2)\xr{\alpha} BF(n).
$$
Using these structure maps we view the above $(\sqcup\downarrow\mathbf n)$-diagram as a diagram over $BF(n)$, and since the Thom space functor preserves colimits by \cite{LMS}, Propositions 1.1 and 1.4, we get the homeomorphism
$$
T(f\boxtimes g)(n)\cong \colim_{(\sqcup\downarrow\mathbf n)}T(\alpha(f,g)).
$$ 
Furthermore, since pullback commutes with topological realization and fibre-wise smash products, we have an isomorphism
$$
\alpha(f,g)^*V(n)\cong S^{n-\alpha}\bar\wedge f^*_{n_1}V(n_1)\bar\wedge g^*_{n_2}V(n_2) 
$$ 
of sectioned spaces over $BF(n)$, hence a homeomorphism of the associated Thom spaces
$$
T(\alpha(f,g))\cong S^{n-\alpha}\wedge T(f_{n_1})\wedge T(g_{n_2}).
$$
Combining the above, we get the homeomorphism
$$
T(f)\wedge T(g)(n)\cong\colim_{\alpha}S^{n-\alpha}\wedge T(f_{n_1})\wedge T(g_{n_2})
\cong T(f\boxtimes g)(n),
$$
specifying the required isomorphism of symmetric spectra.
\end{proof}

\begin{corollary}\label{ringspectrumcorollary}
If $X$ is an $\mathcal I$-space monoid and 
$f\co X\to BF$ a monoid morphism, then $T(f)$
is a symmetric ring spectrum which is commutative if $X$ is.\qed
\end{corollary}
Recall that the \emph{tensor} of an unbased space $K$ with a symmetric spectrum $X$ is defined by the level-wise smash product $X\wedge K_+$. Similarly, the tensor of $K$ with an 
$\I$-space $X$ is defined by the level-wise product $X\times K$. For an object $(X,f)$ in $\I\mathcal U/BF$, the tensor is given by $(X\times K,f\circ\pi_X)$, where $\pi_X$ denotes the projection onto $X$. We refer to \cite{Bo}, Chapter 6, for a general discussion of tensors in enriched categories. 
\begin{proposition}\label{Thomcolimit}
The symmetric Thom spectrum functor preserves colimits and tensors with unbased spaces. 
\end{proposition} 
\begin{proof}
The first statement follows \cite{LMS}, Proposition 1.1 and Corollary 1.4, which combine to show that the Thom space functor preserves colimits. The second claim is that 
$T(f\circ \pi_X)$  is isomorphic to $T(f)\wedge K_+$ which follows directly from the definition.
\end{proof}

\section{Lifting space level data to $\I$-spaces} \label{rectificationsection}
The homotopy colimit construction induces a functor 
\begin{equation}\label{hocolimU/BF}
\hocolim_{\I}\co \I\mathcal U/BF\to \mathcal U/BF_{h\I},\quad (X\to BF)
\mapsto (X_{h\I}\to BF_{h\I}),
\end{equation}
where, given an $\I$-space $X$, we write $X_{h\I}$ for its homotopy colimit. Our first task in this section is to verify that this is the left adjoint in a Quillen equivalence. 

\subsection{The right adjoint of $\hocolim_{\I}$}\label{rightadjointsection}
Recall first that the homotopy colimit functor $\I\U\to \U$ has a right adjoint that to a space $Y$ associates the $\I$-space $\mathbf n\mapsto \Map(B(\mathbf n\downarrow\I),Y)$. Here $(\mathbf n\downarrow \I)$ denotes the category of objects in $\I$ under $\mathbf n$. We refer to 
\cite{BK}, Section XII.2.2, and \cite{HV} for the details of this adjunction. The right adjoint in turn induces a functor 
\[
U\co\U/BF_{h\I}\to \I\U/BF,\quad (X,f)\mapsto (U_f(X),U(f))
\] 
by associating to a map $f\co X\to BF_{h\I}$ the map of $\I$-spaces defined by the upper row in the pullback diagram
\[
\begin{CD}
U_f(X)@>U(f)>>BF\\
@VVV @VVV\\
\Map(B(-\downarrow \I),X)@>>>  \Map(B(-\downarrow \I),BF_{h\I})
\end{CD}
\]
The map on the right is the unit of the adjunction. It is immediate that $U$ is right adjoint to the homotopy colimit functor in (\ref{hocolimU/BF}) and we shall prove in 
Proposition \ref{Q-equivalence} below that the adjunction 
\begin{equation}\label{hocolimU}
\hocolim_{\I}\co
\xymatrix{
\I\mathcal U/BF \ar@<0.5ex>[r] &
\mathcal U/BF_{h\I} \ar@<0.5ex>[l]\thinspace\thinspace\co\!\! U  
}
\end{equation}
is a Quillen equivalence when we give $\U/BF_{h\I}$ the model structure induced by the Quillen model structure on $\U$ and 
$\I\U/BF$ the model structure induced by the $\I$-model structure on $\I\U$ established by Sagave-Schlichtkrull \cite{SS}. Before describing the $\I$-model structure we recall that $\I\U$ has a level model structure in which the weak equivalences and fibrations are defined level-wise.  
Given $d\geq 0$, let $F_d\co \mathcal U\to\I\U$ be the functor that to a space $K$ associates the $\I$-space $F_d(K)=\I(\mathbf d,-)\times K$. Thus, $F_d$ is left adjoint to the evaluation functor that takes an $\I$-space $X$ to $X(d)$. The level structure on $\I\U$ is cofibrantly generated with set of generating cofibrations  
\[
FI=\{F_d(S^{n-1})\to F_d(D^n): d\geq 0, n\geq 0\}
\]
obtained by applying the functors $F_d$ to the set $I$ of generating cofibrations for the Quillen model structure on $\mathcal U$. By a relative cell complex in $\I\U$ we understand a map $X\to Y$ that can be written as the transfinite composition of a sequence of maps
\[
X=Y_0\to Y_1\to Y_2\to\dots \to \colim_{n\geq 0}Y_n=Y
\]
where each map $Y_n\to Y_{n+1}$ is the pushout of a coproduct of generating cofibrations. It follows from the general theory for cofibrantly generated model categories that a cofibration in $\I\U$ is a retract of a cell complex. We refer the reader to \cite{Hi}, Section 11.6, for a general discussion of level model structures on diagram categories.   

As explained in Section \ref{liftingintro}, the weak equivalences in the $\I$-model structure on 
$\I\U$, that is, the $\I$-equivalences, are the maps that induce weak homotopy equivalences of homotopy colimits. The cofibrations in the $\I$-model structure are the same as for the level structure and the fibrations can be characterized as the maps having the right lifting property with respect to acyclic cofibrations. Again, the $\I$-model structure is cofibrantly generated with $FI$ the set of generating cofibrations. There also is an explicit description of the generating acyclic cofibrations and the fibrations but we shall not need this here. 
\begin{lemma}
The adjunction in (\ref{hocolimU}) is a Quillen adjunction
\end{lemma}
\begin{proof}
We claim that $\hocolim_{\I}$ preserves cofibrations and acyclic cofibrations. By definition,
 $\hocolim_{\I}$ preserves weak equivalences in general and the first claim therefore implies the second. For the first claim it suffices to show that $\hocolim_{\I}$ takes the generating cofibrations for $\I\U$ to cofibrations in $\U$. The homotopy colimit of a map $F_d(S^{n-1})\to F_d(D^n)$ may be identified with the map
 \[
 B(\mathbf d\downarrow \I)\times S^{n-1}\to B(\mathbf d\downarrow \I)\times D^n
 \] 
 and the claim follows since $B(\mathbf d\downarrow \I)$ is a cell complex.
\end{proof}
In preparation for the proof that the above Quillen adjunction is in fact a Quillen equivalence we make some general comments on homotopy colimits of $\I$-spaces. 
In general, given an $\I$-space $X$, the homotopy type of $X_{h\I}$ may be very different from that of $X_{h\mathcal N}$. However, if the underlying $\mathcal N$-space is convergent, then the natural map $X_{h\mathcal N}\to X_{h\I}$ is a weak homotopy equivalence by the following lemma due to B\"okstedt; see \cite{M}, Lemma 2.3.7, for a published version. 
\begin{lemma}[\cite{Bok}]\label{Blemma}
Let $X$ be an $\I$-space and suppose that there exists an unbounded non-decreasing sequence of integers $\lambda_m$ such that any morphism $\mathbf m\to\mathbf n$ in $\I$ induces a $\lambda_m$-connected map $X(m)\to X(n)$. Then the inclusion $\{\mathbf n\}\to \I$ induces a $(\lambda_n-1)$-connected map $X(n)\to X_{h\I}$ for all $n$. \qed
\end{lemma}
The structure maps $F(m)\to F(n)$ are $(m-1)$-connected by the Freudentahl suspension theorem and consequently the induced maps $BF(m)\to BF(n)$ are $m$-connected. Thus, the proposition applies to the $\I$-space $BF$ and we see that the canonical map $BF(n)\to BF_{h\I}$ is 
$(n-1)$-connected. This map can be written as the composition 
\[
BF(n)\to \Map(B(\mathbf n\downarrow \I),BF_{h\I})\to BF_{h\I} 
\]
where the first map is the unit of 
the adjunction and the second map is defined by evaluating at the vertex represented by the initial object. Since the second map is clearly a homotopy equivalence it follows that the first map is also $(n-1)$-connected.

\begin{proposition}\label{Q-equivalence}
The adjunction (\ref{hocolimU}) is a Quillen equivalence.
\end{proposition}
\begin{proof}
Given a cofibrant object $f\co X\to BF$ in $\I\U/BF$ and a fibrant object $g\co Y\to BF_{h\I}$ in $\U/BF_{h\I}$  we must show that a morphism $\phi\co X_{h\I}\to Y$ of spaces over $BF_{h\I}$ is a weak homotopy equivalence if and only if the adjoint $\psi\co X\to U_g(Y)$ is an $\I$-equivalence of $\I$-spaces over $BF$. The maps $\phi$ and  $\psi$ are related by the commutative diagram 
\[
\xymatrix{
X_{h\I}\ar[rr]^{\psi_{h\I}} \ar[dr]_{\phi} & & U_g(Y)_{h\I}\ar[dl]^{\varepsilon}\\
& Y &
}
\]
where $\varepsilon$ denotes the counit for the adjunction. It therefore suffices to show that 
$\varepsilon$ is a weak homotopy equivalence. The assumption that $(Y,g)$ be a fibrant object means that $g$ is a fibration and the pullback diagram
\[
\begin{CD}
U_g(Y)(n) @>>> BF(n)\\
@VVV @VVV \\
\Map(B(\mathbf n\downarrow \I), Y)@>>> \Map(B(\mathbf n\downarrow \I),BF_{h\I})
\end{CD}
\] 
used to define $U_g(Y)$ is therefore homotopy cartesian. By the remarks following Lemma \ref{Blemma} it follows that the vertical maps are $(n-1)$-connected. The counit $\varepsilon$ admits a factorization 
\[
U_g(Y)_{h\I}\to \Map(B(-\downarrow \I),Y)_{h\I}\to Y
\]  
where the first map is a weak homotopy equivalence by the above discussion and the second map is a weak homotopy equivalence since $B(-\downarrow \I)$ is level-wise contractible. This completes the proof.
\end{proof}

The functor $U$ is only homotopically well-behaved when applied to fibrant objects. We define a (Hurewicz) fibrant replacement functor $\Gamma$ on $\U/BF_{h\I}$ as in (\ref{fibrantreplacement}) (replacing $BF(n)$ by $BF_{h\I}$) and we write $U'$ for the composite functor $U\circ \Gamma$. This is up to natural homeomorphism the same as the functor obtained by evaluating the homotopy pullback instead of the pullback in the diagram defining $U_f(X)$. 

\subsection{The $\I$-space lifting functor $R$}\label{Rliftingsection}
As discussed in Section \ref{liftingintro}, the functor $U$ does not have all the properties one may wish when constructing Thom spectra from maps to $BF_{h\I}$. In this section we introduce the $\I$-space lifting functor $R$ and we establish some of its properties. 
Given a space $X$ and a map $f\co X\to BF_{h\I}$, we shall view this as a map of constant 
$\I$-spaces. In order to lift it to a map with target $BF$, consider the $\I$-space $\overline{BF}$ defined by
$$
\overline{BF}(n)=\hocolim_{(\I\downarrow \mathbf n)}BF\circ \pi_n,
$$
where $\pi_n\co (\I\downarrow\mathbf n)\to\I$ is the forgetful functor that maps an object 
$\mathbf m\to\mathbf n$ to $\mathbf m$. By definition, $\overline{BF}$ is the homotopy left Kan extension of $BF$ along the identity functor on $\I$, see Appendix \ref{homotopyKansection}. Since the identity on $\mathbf n$ is a terminal object in $(\I\downarrow \mathbf n)$ there results a canonical homotopy equivalence $t_n\co\overline{BF}(n)\to BF(n)$ for each $n$. 
\begin{lemma}\label{barBFlemma}
The map $\pi_n\co\overline{BF}(n)\to BF_{h\I}$ induced by the functor $\pi_n$ is $(n-1)$-connected.
\end{lemma}
\begin{proof}
The homotopy equivalence $t_n$ has a section 
induced by the inclusion of the terminal object in $(\I\downarrow\mathbf n)$, such that the canonical map 
$BF(n)\to BF_{h\I}$ factors through $\overline{BF}(n)$. The result therefore follows from Lemma \ref{Blemma} and the above discussion.
\end{proof}
Consider now the diagram of $\I$-spaces
$$
BF_{h\I}\xl{\pi} \overline{BF}\xr{t} BF,
$$
where the right hand map is the level-wise equivalence specified above and the left hand map is induced by the functors $\pi_n$. Here we again view $BF_{h\I}$ as a constant $\I$-space. We define the $\I$-space $R_f(X)$ to be the level-wise homotopy pullback of the diagram of $\I$-spaces
\begin{equation}\label{Rdiagram}
X\xr{f}BF_{h\I}\xl{\pi}\overline{BF},
\end{equation}
that is, $R_f(X)(n)$ is the space
$$
\{(x,\omega,b)\in X\times BF_{h\I}^I\times\overline{BF}(n)
\co \omega(0)=f(x),\ \omega(1)=\pi(b)\}.
$$
Notice, that the two projections $R_f(X)\to X$ and $R_f(X)\to\overline{BF}$ are level-wise Hurewicz fibrations of $\I$-spaces. The functor $R$ is defined by
\begin{equation}\label{Rfunctor} 
R\co \U/BF_{h\I}\to \I\U/BF,\quad (f\co X\to BF_{h\I})\mapsto 
(R(f)\co R_f(X)\to\overline{BF}\xr{t} BF).
\end{equation}
When there is no risk of confusion we write $R(X)$ instead of $R_f(X)$.
\begin{proposition}\label{RTgood}
The $\I$-space $R_f(X)$ is convergent and $R(f)$ is level-wise $T$-good.
\end{proposition}
\begin{proof}
Since $R_f(X)$ is defined as a homotopy pullback, we see from Lemma \ref{barBFlemma} that the map $R_f(X)(n)\to X$ is $(n-1)$-connected for each $n$, hence $R_f(X)$ is convergent. We claim that 
$R(f)$ classifies a well-based quasifibration at each level. In order to see this we first observe that $t^*V(n)$ is a well-based quasifibration over $\overline{BF}(n)$ by Lemma 
\ref{hocolimquasifibration}. Thus, 
$R(f)^*V(n)$ is a pullback of a well-based quasifibration along the Hurewicz fibration 
$R_f(X)(n)\to \overline{BF}(n)$, hence is itself a well-based quasifibration by Lemma \ref{wellbasedfibrationlemma}.
\end{proof} 
 
\begin{proposition}\label{RU'comparison}
There is a natural level-wise equivalence $R_f(X)\xr{\sim} U'_f(X)$ over $BF$. 
\end{proposition} 
\begin{proof}
Given a map $f\co X\to BF_{h\I}$, consider the diagram of $\I$-spaces 
\[
\xymatrix{
X \ar[r]^f \ar[d] & BF_{h\I}\ar[d] & \overline{BF}\ar[l]_{\pi}\ar[d]^t \\
\Map(B(-\downarrow \I),X)\ar[r]^-f & \Map(B(-\downarrow \I),BF_{h\I})& \ar[l] BF 
}
\]
where we view $X$ and $BF_{h\I}$ as constant $\I$-spaces and the corresponding vertical maps are induced by the projection $B(-\downarrow \I)\to *$. The left hand square is strictly commutative and we claim that the right hand square is homotopy commutative. Indeed, with notation as in Appendix \ref{homotopyKansection}, $\overline{BF}$ is the homotopy Kan extension $\mathit{id}^h_*BF$ along the identity functor on $\I$ and the adjoints of the two compositions in the diagram are the two maps $\hocolim_{\I}\overline{BF}\to BF_{h\I}$ shown to be homotopic in Lemma \ref{homotopyKanlemma}. Using the canonical homotopy from that lemma we therefore get a canonical homotopy relating the two composites in the right hand square. The latter homotopy in turn gives rise to a natural maps of the associated homotopy pullbacks, that is, to a natural map $R_f(X)\to U'_f(X)$. Since the vertical maps in the above diagram are level-wise equivalences the same holds for the map of homotopy pullbacks. 
\end{proof} 
 
\begin{corollary}\label{Rhomotopyinverse}
The functors $R$ and $\hocolim_{\I}$ are homotopy inverses in the sense that there is a chain of natural weak homotopy equivalences $R_f(X)_{h\I}\simeq X$ of spaces over $BF_{h\I}$ and a chain of natural $\I$-equivalences $R_{f_{h\I}}(X_{h\I})\simeq X$ of $\I$-spaces over $BF$.
\end{corollary}  
 
\begin{proof}
It follows easily from Proposition \ref{Q-equivalence} and its proof that the functor $U'$ has this property and the same therefore holds for $R$ by Proposition \ref{Rhomotopyinverse}.
\end{proof} 

The functor $R$ has good properties both formally and homotopically.
\begin{proposition}\label{RUequivalence}
The functor $R$ in (\ref{Rfunctor}) takes weak homotopy equivalences over $BF_{h\I}$ to level-wise equivalences over $BF$ and preserves colimits and tensors with unbased spaces.  
\end{proposition}
\begin{proof}
The first statement follows from the homotopy invariance of homotopy pullbacks. In order to verify that $R$ preserves colimits, we first observe that $BF_{h\I}$ is locally equiconnected (the diagonal $BF_{h\I}\to BF_{h\I}\times BF_{h\I}$ is a cofibration) by \cite{Le}, Corollary 2.4. We then view $R_f(X)$ as the pullback of $X$ along the level-wise Hurewicz fibrant replacement 
$\Gamma_{\pi}(\overline{BF})\to BF_{h\I}$ and the result follows from \cite{LMS}, Propositions 1.1 and 1.2, which together state that the pullback functor along a Hurewicz fibration preserves colimits provided the base space is locally equiconnected. 
The last statement about preservation of tensors is the claim that if $K$ is an unbased space and $(X,f)$ an object of $\mathcal U/BF_{h\I}$, then  $R$ takes $(X\times K,f\circ \pi_X)$ to $R_f(X)\times K$; this follows immediately from the definition.
\end{proof}

Combining this result with Proposition \ref{Tconvergent}, Proposition \ref{Thomcolimit} and Proposition \ref{RTgood}, we get the following corollary in which we define the Thom spectrum functor on $\U/BF_{h\I}$ using $R$.

\begin{corollary}\label{TRcorollary}
The Thom spectrum functor 
\begin{equation}\label{ThomR}
T\co \U/BF_{h\I}\xr{R} \I\U/BF\xr{T} \Sp^{\Sigma}
\end{equation}
takes values in the subcategory of well-based, connective and convergent symmetric spectra. It takes weak homotopy equivalences over $BF_{h\I}$ to level-wise equivalences and preserves colimits and tensors with unbased spaces.\qed
\end{corollary}
 
The functor $R$ also behaves well with respect to cofibrations as we explain next. We follow \cite{MMSS} in using the term \emph{$h$-cofibration} for a morphism having the homotopy extension property. Thus, a map $i\co A\to X$ in $\mathcal U$ is an $h$-cofibration if and only if the induced map from the mapping cylinder
\begin{equation}\label{h-cofibration}
X\cup_i(A\times I)\to X\times I
\end{equation}
admits a retraction. By our conventions, this is precisely what we mean by a cofibration of spaces in this paper. Given a base space $B$, a morphism $i$ in $\mathcal U/B$ is an $h$-cofibration if the analogous morphism (\ref{h-cofibration}) admits a retraction in $\mathcal U/B$; we emphasize this by saying that $i$ is a \emph{fibre-wise} $h$-cofibration. These conventions also apply to define $h$-cofibrations in $\I\mathcal U$ and, given an $\I$-space $B$, fibre-wise 
$h$-cofibrations in  $\I\mathcal U/B$ with the corresponding mapping cylinders defined level-wise.  A morphism $i\co A\to X$ in $\Sp^{\Sigma}$ is an $h$-cofibration if the mapping cylinder $X\cup_{i}(A\wedge I_+)$ is a retract of $X\wedge I_+$. 

\begin{proposition}
The functor $R$ takes maps over $BF_{h\I}$ that are cofibrations in $\mathcal U$ to fibre-wise $h$-cofibrations in $\I\mathcal U/BF$ and the Thom spectrum functor (\ref{ThomR}) takes such maps to $h$-cofibrations of symmetric spectra.
\end{proposition}
\begin{proof}
Notice first that we may view $R_f(X)$ as the pullback of $\overline{BF}$ along the fibrant replacement $\Gamma_f(X)\to BF_{h\I}$. Given a morphism $(A,f)\to(X,g)$ in $\mathcal U/BF_{h\I}$ such that $A\to X$ is a cofibration, the induced map $\Gamma_f(A)\to \Gamma_g(X)$ is a fibre-wise $h$-cofibration by \cite{LMS}, IX, Proposition 1.11. Since fibre-wise $h$-cofibrations are preserved under pullback, this in turn implies that $R_f(A)\to R_g(X)$ is a fibre-wise $h$-cofibration over $\overline{BF}$, hence over $BF$. It follows from Proposition \ref{Thomcolimit} that the Thom spectrum functor on $\I\mathcal U/BF$ takes fibre-wise $h$-cofibrations to $h$-cofibrations. Combining this with the above gives the result.
\end{proof}

\subsection{Preservation of monoidal structures}\label{monoidalsection}
Recall from \cite{Mac}, Section XI.2, that given monoidal categories  $(\mathcal A,\Box,1_{\mathcal A})$ and $(\mathcal B,\triangle,1_{\mathcal B})$, a \emph{monoidal functor} 
$\Phi\co \mathcal A\to\mathcal B$ is a functor $\Phi$ together with a morphism $1_{\mathcal B}\to\Phi(1_{\mathcal A})$ and a natural transformation
$$
\Phi(X)\triangle\Phi(Y)\to\Phi(X\Box Y),
$$ 
satisfying the usual associativity and unitality conditions. It follows from the definition that if $A$ is a monoid in $\mathcal A$, then $\Phi(A)$ inherits the structure of a monoid in $\mathcal B$. Since (unbased) homotopy colimits commute with products, we may view 
$
\hocolim_{\I}
$ 
as a monoidal functor $\I\mathcal U\to \mathcal U$ with structure maps
$$
\hocolim_{\I}X\times \hocolim_{\I}Y\cong\hocolim_{\I\times\I}X\times Y\to \hocolim_{\I}X\boxtimes Y
$$
induced by the universal natural transformation
$
X(\mathbf m)\times Y(\mathbf n)\to X\boxtimes Y(\mathbf m\sqcup\mathbf n) 
$
of $\I\times\I$-diagrams. The unit morphism is induced by the inclusion of the initial object 
$\mathbf 0$, thought of as a vertex in $B\I$. Since $BF$ is a monoid in $\I\mathcal U$, $BF_{h\I}$ inherits the structure of a topological monoid. It follows that we may also view $\mathcal U/BF_{h\I}$ as a monoidal category and the following result is then clear from the definition.
\begin{proposition}
The functor $\hocolim_{\I}$ in (\ref{hocolimU/BF}) is monoidal.\qed
\end{proposition}
However, the functor $\hocolim_{\I}$ is not symmetric monoidal, hence does not take commutative monoids in $\I\mathcal U$ to commutative topological monoids. In particular, $BF_{h\I}$ is not a commutative monoid which is already clear from the fact that it is not equivalent to a product of Eilenberg-Mac Lane spaces. We prove in Section \ref{hocolimoperadsection} that $BF_{h\I}$ has a canonical $E_{\infty}$ structure and that more generally   $\hocolim_{\I}$ takes $E_{\infty}$ objects in $\I\mathcal U$ to $E_{\infty}$ spaces.  
\begin{proposition}\label{Rmonoidalprop}
The functor $R$ in (\ref{Rfunctor}) is monoidal. 
\end{proposition}
\begin{proof}
By definition, $R(*)(0)$ is the loop space of $BF_{h\I}$ and we let $*\to R(*)$ be the map of $\I$-spaces that is the inclusion of the constant loop in degree $0$. We must define an associative and unital natural transformation of $\I$-spaces $R(X)\boxtimes R(Y)\to R(X\times Y)$ over 
$BF$. By the universal property of the $\boxtimes$-product, this amounts to an associative and unital natural transformation  of $\I^2$-diagrams
$$
R(X)(\mathbf m)\times R(Y)(\mathbf n)\to R(X\times Y)(\mathbf m\sqcup\mathbf n).
$$ 
The domain is the homotopy pullback of the diagram 
$$
X\times Y\to BF_{h\I}\times BF_{h\I}\leftarrow \overline{BF}(m)\times\overline{BF}(n),
$$
and the target is the homotopy pullback of the diagram 
$$
X\times Y\to BF_{h\I}\leftarrow \overline{BF}(m+n).
$$
The $\I$-space $\overline{BF}$ inherits a monoid structure from that of $BF$ such that
$\pi\co\overline{BF}\to BF_{h\I}$ is a map of monoids. Using these structure maps, we define a map from the first diagram to the second, giving the required multiplication.
\end{proof}
Since the monoids in the monoidal category $\mathcal U/BF_{h\I}$ are precisely the topological monoids over $BF_{h\I}$, this has the following corollary.
\begin{corollary}\label{monoid-ring}
If $X$ is a topological monoid and $f\co X\to BF_{h\I}$ a monoid morphism, then $T(f)$ is a symmetric ring spectrum.\qed
\end{corollary}
This may be reformulated as saying that the Thom spectrum functor preserves the action of the associativity operad whose $k$th space is the symmetric group $\Sigma_k$, see \cite{Ma}, Section 3. More generally, we show in Section \ref{Thomoperadsection} that $T$ preserves all operad actions of operads that are augmented over the Barratt-Eccles operad. 

\subsection{Comparison with the Lewis-May Thom spectrum functor}\label{May-Lewis}
Let as before $BF_{\mathcal N}$ denote the colimit of $BF$ over $\mathcal N$. In this section we recall the Thom spectrum functor on $\mathcal U/BF_{\mathcal N}$ considered in \cite{LMS}, Section IX, and we relate this to our symmetric Thom spectrum functor on 
$\mathcal U/BF_{h\I}$. We shall use the same notation for the $\I$-space $BF$ and its restriction  to an $\mathcal N$-space. The colimit functor $\mathcal N\U/BF\to \U/BF_{\mathcal N}$ has a right adjoint, again denoted $U$, that to an object $f\co X\to BF_{\mathcal N}$ associates the map of $\mathcal N$-spaces defined by the upper row in the pullback diagram
\[
\begin{CD}
U_f(X)@>U(f)>> BF\\
@VVV @VVV\\
X@>f>> BF_{\mathcal N}
\end{CD}
\] 
where the vertical map on the right is the unit of the adjunction relating the 
colimit and the constant functors. Here we view $X$ and $BF_{\mathcal N}$ as constant 
$\mathcal N$-spaces. We again write $U'$ for the functor obtained by composing with the Hurewicz fibrant replacement functor $\Gamma$ on $\U/BF_{\mathcal N}$.
The Thom spectrum functor  considered in \cite{LMS} is the composition
$$
\mathcal U/BF_{\mathcal N}\xr{U} \mathcal N\mathcal U/BF\xr{T}\Sp,
$$
where $T$ is the functor from Section \ref{preliminariessection}. (In the language of \cite{LMS} this is the Thom \emph{prespectrum} associated to $f$.  The authors go on to define a spectrum $M(f)$ with the property that the adjoint structure maps are homeomorphisms, but this will not be relevant for the discussion here). The first step in the comparison to our symmetric Thom spectrum functor on $\U/BF_{h\I}$ is to relate the spaces $BF_{\mathcal N}$ and 
$BF_{h\I}$. Consider the diagram of weak homotopy equivalences
$$
BF_{h\I}\xl{i} BF_{h\mathcal N}\xr{t} BF_{\mathcal N},
$$ 
where $i$ is induced from the inclusion $i\co\mathcal N\to\I$ and $t$ is the canonical projection from the homotopy colimit to the colimit. The former is a weak homotopy equivalence by 
Lemma \ref{Blemma} and the latter is a weak homotopy equivalence since the structure maps are cofibrations. Let us choose a homotopy inverse $j\co BF_{h\I}\to BF_{h\mathcal N}$ of $i$ and a homotopy relating $i\circ j$ to the identity on $BF_{h\I}$. Here we of course use that these spaces have the homotopy type of a CW-complex.  The precise formulation of the comparison will depend on these choices. Let $\zeta$ be the composite homotopy equivalence
$$
\zeta\co BF_{h\I}\xr{j} BF_{h\mathcal N}\xr{t} BF_{\mathcal N}.
$$
In general, given a map $\phi\co B_1\to B_2$ in $\mathcal U$, we write $\phi_*\co \U/B_1\to \U/B_2$ for the functor defined by post-composing with $\phi$.

\begin{lemma}\label{B_1B_2}

Suppose that  $\phi$ and $\psi$ are maps from $B_1$ to $B_2$ that are homotopic by a homotopy $h\co B_1\times I\to B_2$. Then the functors
$\phi_*$ and $\psi_*$ from $\U/B_1$ to  $\U/B_2$ 
are related by a chain of natural weak homotopy equivalences depending on $h$.
\end{lemma}
\begin{proof}
Let $h_*\co \U/B_1\to \U/B_2$ be the functor that takes $f\co X\to B_1$ to  
$$
X\times I\xr{f\times I}B_1\times I\xr{h}B_2.
$$ 
The two endpoint inclusions of $X$ in $X\times I$ then give rise to the natural weak homotopy equivalences $\phi_*\to h_*\leftarrow \psi_*$.
\end{proof}

Applied to the homotopy relating $i\circ j$ to the identity on $BF_{h\I}$ this result gives a chain of natural weak homotopy equivalences relating the composite functor
\[
\U/BF_{h\I}\xr{j_*}\U/BF_{h\mathcal N}\xr{i_*} \U/BF_{h\I}
\]
to the identity on $\U/BF_{h\I}$.

\begin{lemma}\label{RFGamma}
The two compositions in the diagram
$$
\begin{CD}
\U/BF_{h\I}@>U'>> \I\U/BF\\
@VV\zeta_* V @VVi^* V\\
\U/BF_{\mathcal N} @>U'>>\mathcal N\U/BF
\end{CD}
$$ 
are related by a chain of natural level-wise equivalences.
\end{lemma}  
\begin{proof}
We shall interpolate between these functors by relating both to the $\mathcal N$-space analogue of the functor $U'$ on $\U/BF_{h\I}$. Thus, given $f\co X\to BF_{h\mathcal N}$, the diagram of 
$\mathcal N$-spaces
\[
\Map(B(-\downarrow \mathcal N), X) \xr{f}
\Map(B(-\downarrow \mathcal N), BF_{h\mathcal N}) \xl{}BF
\]
is related by evident chains of term-wise level equivalences to the diagrams 
\[
\Map(i^*B(-\downarrow \I),X)\xr{i\circ f}
\Map(i^*B(-\downarrow\I),BF_{h\I}) \xl{}BF
\]
and
\[
X\xr{t\circ f}BF_{\mathcal N}\xl{} BF.
\] 
Evaluating the homotopy pullbacks of these diagrams we get a chain of natural level-wise equivalences relating the two compositions in the diagram
\[
\xymatrix{
\U/BF_{\mathcal N}\ar[d]^{U'} &\U/BF_{h\mathcal N} \ar[l]_{t_*}\ar[r]^{i_*} & 
\U/BF_{h\I}\ar[d]^{U'}\\
\mathcal N\U/BF & &\I\U/BF\ar[ll]_{i^*}. 
}
\]
By the remarks following Lemma \ref{B_1B_2} we therefore get a chain of natural transformations  
\begin{equation}\label{Lewis-Maychain}
i^*\circ U'\sim i^*\circ U'\circ i_*\circ j_*\sim U'\circ t_*\circ j_*\sim U'\circ \zeta_*,
\end{equation}
each of which is a level-wise weak homotopy equivalence.
\end{proof}

We can now compare our symmetric Thom spectrum functor to the Lewis-May Thom spectrum functor on $\U/BF_{\mathcal N}$.  Since the functors $TR$ and $TU'$ on $\U/BF_{h\I}$ are 
level-wise equivalent by Proposition \ref{RU'comparison}, it suffices to consider $TU'$.

\begin{proposition}\label{Lewis-Mayprop}
The two compositions in the diagram 
\[
\begin{CD}
\mathcal U/BF_{h\I} @>TU'>> \Sp^{\Sigma}\\
@VV \zeta_* V @VVV\\
\mathcal U/BF_{\mathcal N} @>T U'>> \Sp
\end{CD}
\]
are related by a chain of level-wise equivalences. 
\end{proposition}

\begin{proof}
The diagram in question is obtained by composing the diagram in Proposition \ref{RFGamma} with the commutative diagram (\ref{INThom}). Since the chain of weak homotopy equivalences in   (\ref{Lewis-Maychain}) is contained in the full subcategory of level-wise $T$-good objects in $\mathcal N\U/BF$, applying $T$ gives a chain of level-wise equivalences.
\end{proof}

\section{Homotopy invariance of symmetric Thom spectra}\label{invariancesection}
In this section we prove the homotopy invariance result stated in Theorem \ref{introinvariance} and we show how  the proof can be modified to give the $\mathcal N$-space analogue in 
Theorem \ref{NThominvariance}. As for the Thom space functor, the symmetric Thom spectrum functor is not homotopically well-behaved on the whole domain category 
$\I\U/BF$. We define a level-wise Hurewicz fibrant replacement functor on $\I\U/BF$ by 
applying the functor $\Gamma$ in (\ref{fibrantreplacement}) at each level.
\begin{definition}\label{T-good-definition}
An object $(X,f)$ in $\I\U/BF$ is \emph{$T$-good} if the canonical map 
$T(f)\to T(\Gamma(f))$ is a stable equivalence (a weak equivalence in the stable model structure) of symmetric spectra.
\end{definition}
 
 As before we say that $(X,f)$ is level-wise $T$-good if $T(f)\to T(\Gamma(f))$ is a level-wise equivalence. The first step in the proof of Theorem \ref{introinvariance} is to generalize the definition of 
$\overline{BF}$ to any $\I$-space $X$ by associating to $X$ the $\I$-space $\overline X$ 
defined by
$$
\overline X(n)=\hocolim_{(\I\downarrow\mathbf n)}X\circ \pi_n.
$$
We then have a diagram of $\I$-spaces
$$
X_{h\I}\xl{\pi} \overline{X}\xr{t}X,
$$ 
where we view $X_{h\I}$ as a constant $\I$-space. The map $t$ is a level-wise equivalence and $\pi$ is an $\I$-equivalence by Lemma \ref{homotopyKanequivalence}. If $f\co X\to BF$ is a map of 
$\I$-spaces, then we have a commutative diagram
$$
\begin{CD}
\overline X@>\bar f >>\overline{BF}\\
@VV\pi V @VV \pi V\\
X_{h\I}@>f_{h\I}>> BF_{h\I},
\end{CD} 
$$
hence there is an induced morphism
\begin{equation}\label{barXR}
(\overline X,t\circ \bar f)\to (R_{f_{h\I}}(X_{h\I}), R(f_{h\I}))
\end{equation}
of $\I$-spaces over $BF$.
\begin{proposition}\label{Tbarf}
Applying $T\circ \Gamma$ to the morphism (\ref{barXR}) gives a stable equivalence of symmetric spectra.
\end{proposition}
In order to prove this proposition we shall make use of the level model structure on $\I\U$ recalled in Section \ref{rightadjointsection}. Let $F_d\co \U\to \I\U$ be the functor defined in that section and let us write $F_d(u)\co F_d(K)\to BF$ for the map of $\I$-spaces 
associated to a map of spaces $u\co K\to BF(d)$.

\begin{lemma}\label{ugood}
If $u$ is a Hurewicz fibration, then $F_d(u)$ is level-wise $T$-good.
\end{lemma}
\begin{proof}
The pullback of $V(n)$ along $F_d(u)$ is isomorphic to the coproduct of the pullbacks along $u$ of the fibre-wise suspensions $S^{n-\alpha}\bar\wedge V(d)$ over $BF(d)$, where $\alpha$ runs through the injective maps $\mathbf d\to\mathbf n$. These are well-based quasifibrations by Proposition \ref{barprop} and since $u$ is a fibration, the same holds for the pullbacks by Lemma 
\ref{wellbasedfibrationlemma} and the claim follows. 
\end{proof}

The idea is to first prove Proposition \ref{Tbarf} for objects of the form $F_d(u)$. 
\begin{lemma}\label{freeThomlemma}
Applying $T\circ \Gamma$ to the map of $\I$-spaces
$
\overline{F_d(K)}\to R(F_d(K)_{h\I})
$
over $BF$ gives a stable equivalence of symmetric spectra.
\end{lemma}
The proof of this requires some preparation. We view $K$ as a space over $BF_{h\I}$ via the map
\begin{equation}\label{utilde}
\tilde u\co K\to BF(d)\to BF_{h\I},
\end{equation}
where the second map is induced by the inclusion of $\{\mathbf d\}$ in $\I$.  

\begin{lemma}\label{Fhocolim}
There is a weak homotopy equivalence $K\to F_d(K)_{h\I}$ of spaces over $BF_{h\I}$.  
\end{lemma}
\begin{proof}
By definition of the homotopy colimit we may identify $F_d(K)_{h\I}$ with 
$B(\mathbf d\downarrow \I)\times K$, where $(\mathbf d\downarrow\I)$ is the category of objects in $\I$ under $\mathbf d$. Since this category has an initial object its classifying space is contractible and the result follows. 
\end{proof}

In the case of the $\I$-space $F_d(K)$, the level-wise equivalence 
$t\co \overline{F_d(K)}\to F_d(K)$ has a section induced by the canonical map 
$K\to \overline{F_d(K)}(d)$. Using this, we get a commutative diagram in 
$\I\mathcal U/BF$,
\begin{equation}\label{freeRdiagram}
\begin{CD}
F_d(K)@>\sim >> \overline{F_d(K)}\\
@VVV @VVV\\
R(K)@>\sim >> R(F_d(K)_{h\I}).
\end{CD}
\end{equation}
The upper horizontal map is a level-wise equivalence since $t$ is and the lower horizontal map is a level-wise equivalence by the above lemma. Thus, in order to prove Lemma 
\ref{freeThomlemma}, we may equally well consider the vertical map on the left hand side of the diagram.

Given a based space $T$, let $F_d^S(T)$ be the symmetric spectrum 
$\I_S(\mathbf d,-)\wedge T$. The functor $F_d^S$ so defined is left adjoint to the functor 
$\Sp^{\Sigma}\to \mathcal T$ that takes a symmetric spectrum to its $d$th space, see \cite{MMSS}.  In particular it follows from the definition that $F_0^S(T)$ is the suspension spectrum of $T$. Notice also that the Thom spectrum $T(F_d(u))$ associated to $F_d(u)$ may be identified with $F_d^S(T(u))$, where as usual $T(u)$ denotes the Thom space of the map $u$. Let $T(\tilde u)$ be the symmetric Thom spectrum of the map $\tilde u$ in (\ref{utilde}) and let 
$\Sigma_L^dT(\tilde u)$ be the left shift by 
$\mathbf d$, that is, the composition of $T(\tilde u)$ with the concatenation functor $\I_S\to\I_S$, $\mathbf n\mapsto \mathbf d\sqcup\mathbf n$. Thus, the $n$th space of 
$\Sigma^d_LT(\tilde u)$ is $T(\tilde u)(\mathbf d\sqcup\mathbf n)$ with $\Sigma_n$ acting via the inclusion $\Sigma_n\to\Sigma_{d+n}$ induced by $\mathbf n\mapsto \mathbf d\sqcup\mathbf n$.
The condition that $u$ be a Hurewicz fibration in the following lemma is unnecessarily restrictive, but the present formulation is sufficient for our purposes. 
\begin{lemma}\label{ulemma}
If $u$ is a Hurewicz fibration, then the canonical map of spaces $T(u)\to T(\tilde u)(d)$ induces a $\pi_*$-isomorphism $F_0^S(T(u))\to\Sigma_L^dT(\tilde u)$.
\end{lemma}
\begin{proof}
 In spectrum degree $n$ this is the map of Thom spaces induced by the map 
$$
K\to  R_{\tilde u}(K)(\mathbf d)\to R_{\tilde u}(K)(\mathbf d\sqcup\mathbf n),
$$
viewed as a map of $T$-good spaces over $BF(d+n)$. This is also a map of spaces over $K$ via the projection $R_{\tilde u}(K)\to K$, and it therefore follows from the proof of Proposition 
\ref{RTgood} that its connectivity tends to infinity with $n$. The result then follows from Lemma \ref{basicthomlemma}.
\end{proof}

We shall prove Lemma \ref{freeThomlemma} using the detection functor $D$ from \cite{Sh}. We recall that this functor associates to a symmetric spectrum $T$ the symmetric spectrum $DT$ whose $n$th space is the based homotopy colimit
$$
DT(n)=\hocolim_{\mathbf m\in \I}\Omega^m(T(m)\wedge S^n).
$$ 
By \cite{Sh}, Theorem 3.1.2, a map of (level-wise well-based) symmetric spectra $T\to T'$ is a stable equivalence if and only if the induced map $DT\to DT'$ is a $\pi_*$-isomorphism. There is a closely related functor $T\mapsto MT$, where $MT$ is the symmetric spectrum with $n$th space
$$
MT(n)=\hocolim_{\mathbf m\in \I}\Omega^m(T(\mathbf m\sqcup \mathbf n)).
$$
Thus, $MT$ is the homotopy colimit of the $\I$-diagram of symmetric spectra 
$\mathbf m\mapsto \Omega^m(\Sigma^m_LT)$.
There is a canonical map $DT\to MT$, which is a level-wise equivalence if $T$ is convergent and level-wise well-based. 

\medskip
\noindent\textit{Proof of Lemma \ref{freeThomlemma}.}
We claim that applying $T\circ\Gamma$ to the vertical map on the left hand side of  
(\ref{freeRdiagram}) gives a stable equivalence, and for this we may assume without loss of generality that $u$ is a Hurewicz fibration. Then $F_d(u)$ is $T$-good by Lemma \ref{ugood} 
and since $R(\tilde u)$ is $T$-good by Proposition \ref{RTgood}, it suffices to show that 
$T(F_d(u))\to T(\tilde u)$ is a stable equivalence. 
Furthermore, by  \cite{MMSS}, Theorem 8.12, it is enough to show that this map is a stable equivalence after smashing with $S^d$ and by the above remarks this in turn follows if applying $D$ gives a $\pi_*$-isomorphism. We identify  $T(F_d(u))$ with $F_d^S(T(u))$ and
claim that there is a commutative diagram
$$
\begin{CD}
F_0^S(T(u))@>\sim>> \Sigma_L^dT(\tilde u)@>\sim >> 
\Omega^d(S^d\wedge \Sigma^d_LT(\tilde u))\\
@VV\sim V @. @VV\sim V\\
D(S^d\wedge F^S_d(T(u)))@>>> D(S^d\wedge T(\tilde u)) @>\sim >> M(S^d\wedge T(\tilde u)),
\end{CD}
$$
where the maps are $\pi_*$-isomorphisms as indicated. The vertical map on the left hand side is induced by the space-level map 
$$
T(u)\to F_d^S(T(u))(d)\to \Omega^d(S^d\wedge F_d^S(T(u))(d))\to 
D(S^d\wedge F^S_d(T(u)))(0).
$$
It is a fundamental property of the model structure on $\Sp^{\Sigma}$ that the induced map of symmetric spectra is a $\pi_*$-isomorphism, see the proof of \cite{Sh}, Lemma 3.2.5. The first map in the upper row is the stable equivalence from Lemma \ref{ulemma}, and the remaining indicated arrows are $\pi_*$-isomorphisms since $T(\tilde u)$ is connective and convergent. This proves the claim.\qed

\medskip
We now wish to prove Proposition \ref{Tbarf} by an inductive argument based on the filtration 
\begin{equation}\label{cellfiltration}
\emptyset=X_0\to X_1\to X_2\to \dots\to \colim_{n}X_n=X
\end{equation}
of a cell complex $X$ in $\I\U$, cf.\ the discussion of the level model structure in Section \ref{rightadjointsection}. 
In order to carry out the induction step, we need to ensure that the induced maps of Thom spectra are $h$-cofibrations in the sense of Section \ref{Rliftingsection}.  
The following is the $\I$-space analogue of \cite{LMS}, IX, Lemma 1.9 and Proposition 1.11. The proof is essentially the same as in the space-level case.
\begin{proposition}
The functor $\Gamma$ on $\I\mathcal U/BF$ preserves colimits and takes morphisms in $\I\mathcal U/BF$ that are $h$-cofibrations in $\I\mathcal U$ to fibre-wise $h$-cofibrations.\qed
\end{proposition}
Since the symmetric Thom spectrum functor on $\I\mathcal U/BF$ preserves colimits and takes fibre-wise $h$-cofibrations to $h$-cofibrations by Proposition \ref{Thomcolimit}, this has the following consequence.
\begin{proposition}
The composite functor $T\circ\Gamma$ preserves colimits and takes morphisms in $\I\mathcal U/BF$ that are $h$-cofibrations in $\I\mathcal U$ to $h$-cofibrations of symmetric spectra.  \qed
\end{proposition}

\begin{proof}[Proof of Proposition \ref{Tbarf}.]
Using the level model structure we may choose a cofibrant $\I$-space $X'$ and a level-wise equivalence $X'\to X$, hence it suffices to consider the case where $X$ is a cofibrant $\I$-space. Then $X$ is a retract of a cell complex which we may view as a cell complex over $BF$ via the retraction. By functoriality we are thus reduced to the case where $X$ is a cell complex with a filtration by $h$-cofibrations as in (\ref{cellfiltration}). In order to handle this case we use that both functors in (\ref{barXR}) preserve colimits and tensors with unbased spaces, hence they also preserve (not necessarily fibre-wise) $h$-cofibrations. Applying the functor $T\circ\Gamma$, we see that both functors in the proposition preserve colimits and take $h$-cofibrations of $\I$-spaces over $BF$ to $h$-cofibrations of symmetric spectra. We prove by induction that the result holds for each of the $\I$-spaces $X_n$ in the filtration. By definition, $X_{n+1}$ is a pushout of a diagram of the form
$
B\leftarrow A\to X_n,
$
where $A\to B$ is a coproduct of generating cofibrations, hence in particular an $h$-cofibration. We view this as a diagram of $\I$-spaces over $BF$ via the inclusion of $X_{n+1}$ in $X$ and get a diagram of Thom spectra
$$
\begin{CD}
T\Gamma(\overline{X}_n)@<<< T\Gamma(\overline{A})@>>> 
T\Gamma(\overline{B})\\\
@VVV @VVV @VVV\\
T\Gamma(R((X_n)_{h\I}))@<<<T\Gamma(R(A_{h\I}))@>>> T\Gamma(R(B_{h\I})),
\end{CD}
$$    
such that the map for $X_{n+1}$ is the induced map of pushouts. By the above discussion it follows that the horizontal maps on the right hand side of the diagram are $h$-cofibrations and the vertical maps are stable equivalences by Lemma \ref{freeThomlemma} and the induction hypothesis. Consequently the map of pushouts is also a stable equivalence, see \cite{ MMSS}, Theorem 8.12.
\end{proof}

\begin{proof}[Proof of Theorem \ref{introinvariance}]
We prove that applying the functor $T\circ\Gamma$ to an $\I$-equivalence $X\to Y$ over $BF$ gives a stable equivalence of symmetric spectra. Consider the commutative diagram
$$
\begin{CD}
X@<<<\overline{X}@>>> R(X_{h\I})\\
@VVV @VVV @VVV\\
Y@<<< \overline{Y} @>>> R(Y_{h\I})
\end{CD}
$$
of $\I$-spaces over $BF$. Applying $T\circ\Gamma$ to this diagram we get a diagram of symmetric spectra where the horizontal maps are stable equivalence by Proposition \ref{Tbarf} and the fact that $T\circ\Gamma$ preserves level-wise equivalences. The result now follows from Corollary \ref{TRcorollary} which ensures that the map $R(X_{h\I})\to R(Y_{h\I})$ induces a stable equivalence. 
\end{proof}

Notice, that as a consequence of the theorem, the composite functor $T\circ\Gamma$ is a homotopy functor on $\I\U/BF$ in the sense that it takes $\I$-equivalences to stable equivalences.

\subsection{The proof of Theorem of \ref{NThominvariance}}\label{NTinvariancesection}
The proof of Theorem \ref{NThominvariance} is similar to but simpler than the proof of Theorem 
\ref{introinvariance}. We first introduce a functor
\[
R^{\mathcal N}\co \U/BF_{h\mathcal N}\to \mathcal N\U/BF,
\]
which is the $\mathcal N$-space analogue of the functor $R$. Let us temporarily write 
$\overline{BF}$ for the homotopy Kan extension of the $\mathcal N$-space $BF$ along the identity functor of $\mathcal N$, 
that is, 
$$
\overline{BF}(n)=\hocolim_{(\mathcal N\downarrow\mathbf n)} BF\circ\pi_n,
$$  
where $\pi_n$ is the forgetful functor $(\mathcal N\downarrow\mathbf n)\to\mathcal N$. 
Given a map $f\co X\to BF_{h\mathcal N}$, we define $R^{\mathcal N}_f(X)$ to be the level-wise homotopy pullback of the diagram of $\mathcal N$-spaces
$$
X\xr{f}BF_{h\mathcal N}\xl{\pi}\overline{BF}, 
$$
and we define $R^{\mathcal N}(f)$ to be the composite map of $\mathcal N$-spaces
$$
R^{\mathcal N}(f)\co R^{\mathcal N}_f(X)\to \overline{BF}\xr{t} BF.
$$    
Exactly as in the $\I$-space case there is a map of $\mathcal N$-spaces 
\begin{equation}\label{tildeXR}
(\overline{X},t\circ \overline f)\to (R^{\mathcal N}_{f_{h\mathcal N}}(X_{h\mathcal N}), 
R^{\mathcal N}(f_{h\mathcal N}))
\end{equation}
over $BF$, where we again use the (temporary) notation $\overline X$ for the homotopy Kan extension along the identity on $\mathcal N$.
Theorem \ref{NThominvariance} then follows from the following proposition in the same way that Theorem \ref{introinvariance} follows from Proposition \ref{Tbarf}. 

\begin{proposition}\label{Ttildef}
Applying $T\circ\Gamma$ to (\ref{tildeXR}) gives a stable equivalence of spectra. 
\end{proposition}
In order to prove this we first consider the $\mathcal N$-spaces $F_d(K)$ defined by $\mathcal N(\mathbf d,-)\times K$, where $\mathbf d$ is an object in $\mathcal N$ and $K$ is a space. Given a map $u\co K\to BF(d)$, we have the following $\mathcal N$-space analogue of 
(\ref{freeRdiagram}),
$$
\begin{CD}
F_d(K)@>>> \overline{F_d(K)}\\
@VVV @VVV\\
R(K) @>>> R(F_d(K)_{h\mathcal N}).
\end{CD}
$$
However, in contrast to the $\I$-space setting, this is a diagram of convergent $\mathcal N$-spaces and the connectivity of the maps in degree $n$ tends to infinity with $n$. 
Thus, the $\mathcal N$-space analogue of Lemma \ref{freeThomlemma} holds with a simpler proof. 

\medskip
\noindent\textit{Proof of Proposition \ref{Ttildef}.}
We use that $\mathcal N\U$ has a cofibrantly generated level model structure and as in the 
$\I$-space case we reduce to the case of a cell complex. Using that the functors in (\ref{tildeXR}) preserve colimits and $h$-cofibrations, the inductive argument used in the proof of Proposition \ref{Tbarf} then also applies in the $\mathcal N$-space setting.\qed
 
\section{Preservation of operad actions}\label{Thomoperadsection}
Let $\mathcal C$ be an operad as defined in \cite{Ma} and notice that $\mathcal C$ defines a monad $C$ on the symmetric monoidal category $\I\U$ in the usual way by letting
$$
C(X)=\coprod_{k=0}^{\infty}\mathcal C(k)\times_{\Sigma_k}X^{\boxtimes k}.
$$
We define a \emph{$\mathcal C$-$\I$-space} to be an algebra for this monad and write 
$\I\mathcal U[\mathcal C]$ for the category of such algebras. More explicitly, a $\mathcal C$-$\I$-space is an $\I$-space $X$ together with a sequence of maps of $\I$-spaces 
$$
\theta_k\co\mathcal C(k)\times X^{\boxtimes k}\to X,
$$
satisfying the associativity, unitality and equivariance relations listed in \cite{Ma}, Lemma 1.4.  
By the universal property of the $\boxtimes$-product, $\theta_k$ is determined by a natural transformation of $\I^k$-diagrams 
\begin{equation}\label{Ioperadaction}
\theta_k\co \mathcal C(k)\times X(n_1)\times\dots\times X(n_k)\to 
X(n_1+\dots+ n_k)
\end{equation}
and the equivariance condition amounts to the commutativity of the diagram
$$
\begin{CD}
\mathcal C(k)\times X(n_1)\times\dots\times X(n_k)@>\theta_k\circ(\sigma\times\text{id})>>
X(n_1+\dots+n_k)\\
@VV \text{id}\times \sigma V  @VV{\sigma(n_1,\dots,n_k)_*}V\\
\mathcal C(k)\times X(n_{\sigma^{-1}(1)})\times \dots\times X(n_{\sigma^{-1}(k)})
@>\theta_k>>
X(n_{\sigma^{-1}(1)}+\dots+n_{\sigma^{-1}(k)})
\end{CD}
$$
for all elements $\sigma$ in $\Sigma_k$. Here $\sigma$ permutes the factors on the left hand side of the diagram and 
$\sigma(n_1,\dots,n_k)$ denotes the permutation of 
$\mathbf n_1\sqcup\dots\sqcup\mathbf n_k$ that permutes the $k$ summands as 
$\sigma$ permutes the elements of $\mathbf k$.  
As defined in \cite{Ma}, the $0$th space of $\mathcal C$ is a one-point space, so 
that $\theta_0$ specifies a base point of $X$. 
Notice, that an action of the one-point operad $*$ on an $\I$-space $X$ is the same thing as a commutative monoid structure on $X$. In this case the projection $\mathcal C\to *$ induces a $\mathcal C$-action on $X$ for any operad $\mathcal C$.  This applies in particular to the commutative $\I$-space monoid $BF$. 

In similar fashion an operad $\mathcal C$ defines a monad $C$ on the category $\Sp^{\Sigma}$ by letting
$$
C(X)=\bigvee_{k=0}^{\infty}\mathcal C(k)_+\wedge_{\Sigma_k}X^{\wedge k}
$$
and we write $\Sp^{\Sigma}[\mathcal C]$ for the category of algebras for this monad. Thus, an object of $\Sp^{\Sigma}[\mathcal C]$ is a symmetric spectrum $X$ together with a sequence of maps of symmetric spectra
$$
\theta_k\co \mathcal C(k)_+\wedge X^{\wedge k}\to X,
$$
satisfying the analogous associativity, unitality and equivariance relations. By the universal property of the smash product, $\theta_k$ is determined by a natural transformation of $\I_S^k$-diagrams,
\begin{equation}\label{naturaloperadaction}
\theta_k\co\mathcal C(k)_+\wedge X(n_1)\wedge\dots\wedge X(n_k)\to X(n_1+\dots+n_k).
\end{equation}
The naturality condition can be formulated explicitly as follows. Given a family of morphisms 
$\alpha_i\co\mathbf m_i\to\mathbf n_i$ in $\mathcal I$ for $i=1,\dots, k$, let $\alpha=\alpha_1\sqcup\dots\sqcup\alpha_k$. Writing $n=n_1+\dots+n_k$ and making the identification 
$$
S^{n_1-\alpha_1}\wedge\dots\wedge S^{n_k-\alpha_k}=S^{n-\alpha},
$$ 
we require that the diagram
$$
\begin{CD}
S^{n-\alpha}\wedge\mathcal C(k)_+\wedge X(m_1)\wedge\dots\wedge X(m_k) 
@>S^{n-\alpha}\wedge\theta_k>>S^{n-\alpha}\wedge X(m_1+\dots+m_k)\\
@VVV @VVV\\
\mathcal C(k)_+\wedge X(n_1)\wedge\dots\wedge X(n_k)@>\theta_k>> X(n_1+\dots+n_k)
\end{CD}
$$
be commutative.
 We now show that the symmetric Thom spectrum functor behaves well with respect to operad actions. Given an operad $\mathcal C$ and a map of $\I$-spaces $f\co X\to BF$, let $C(f)$ be the composite map
$$
C(f)\co C(X)\to C(BF)\to BF.
$$
The following is the analogue in our setting of \cite{LMS}, Theorem IX 7.1. It is a formal consequence of the fact that $T$ is a strong symmetric monoidal functor that preserves colimits and tensors with unbased spaces. 
\begin{proposition}
There is a canonical isomorphism of symmetric spectra
$$
T(C(f))=C(T(f)).
\eqno\qed
$$
\end{proposition}

\begin{corollary}\label{thomoperadcorollary}
The Thom spectrum functor on $\I\mathcal U/BF$ preserves operad actions in the sense that there is an induced functor
$$
T\co\mathcal I\mathcal U[\mathcal C]/BF\to\Sp^{\Sigma}[\mathcal C].
\eqno\qed
$$
\end{corollary}

\subsection{Operad actions preserved by $\hocolim_{\I}$}\label{hocolimoperadsection}
As in Section \ref{liftingintro} we use the notation $\mathcal E$ for the Barratt-Eccles operad. We recall that the $k$th space $\mathcal E(k)$ is the classifying space of the translation category $\tilde \Sigma_k$ that has the elements of $\Sigma_k$ as its objects. A morphism 
$\rho\co\sigma\to\tau$ in $\tilde \Sigma_k$ is an element $\rho\in\Sigma_k$ such that $\rho\sigma=\tau$; see  \cite{Ma1}, Section 4 (but notice that the order of the composition in $\tilde\Sigma_k$ is defined differently here). In the following proposition, $\mathcal C$ denotes an arbitrary operad and $\mathcal E\times \mathcal C$ denotes the product operad whose $k$th space is the product $\mathcal E(k)\times \mathcal C(k)$. 
\begin{proposition}\label{Einftyprop}
The functor $\hocolim_{\I}$ induces a functor 
$$
\hocolim_{\I}\co\I\mathcal U[\mathcal C] \to\mathcal U[\mathcal E\times\mathcal C].
$$
\end{proposition}
\begin{proof}
Let $\mathcal I(X)$ be the topological category whose space of objects is the disjoint union
of the spaces $X(n)$ indexed by the objects $\mathbf n$ in $\I$, and in which a
morphism $(\mathbf m,x)\to(\mathbf n,y)$ is specified by a morphism $\alpha\co\mathbf m\to\mathbf n$ in $\I$ such that $\alpha_*(x)=y$. Then it follows from the definition of the homotopy colimit that $X_{h\I}$ may be identified with the classifying space $B\I(X)$; see 
Appendix \ref{hocolimsection} for details. In the following we shall view the spaces $\mathcal C(k)$ as topological categories with only identity morphisms. For each $k$, consider the functor of topological categories 
$$
\psi_k\co\tilde\Sigma_k\times \mathcal C(k)\times \I(X)^k\to\I(X),
$$
that maps a tuple of objects $\sigma\in \Sigma_k$, $c\in \mathcal C(k)$, and $(\mathbf n_1,x_1),\dots,(\mathbf n_k,x_k)$, to
$$
(\mathbf n_{\sigma^{-1}(1)}\sqcup\dots\sqcup\mathbf n_{\sigma^{-1}(k)},
\sigma(n_1,\dots,n_k)_*\theta_k(c,x_1,\dots,x_k)).
$$
Here $\theta_k$ denotes the $\mathcal C(k)$-action on $X$ and $\sigma(n_1,\dots,n_k)$ is defined as at the beginning of this section. If $\rho\co\sigma\to\tau$ is a morphism in $\tilde\Sigma_k$, then the induced morphism in $\I(X)$ is specified by
$$
\psi_k(\rho)=\rho(n_{\sigma^{-1}(1)},\dots,n_{\sigma^{-1}(k)}),
$$
and if $\vec\alpha$ denotes a $k$-tuple of morphisms in $\I(X)$ whose $i$th component is specified by $\alpha_i\co\mathbf n_i\to\mathbf m_i$, then the induced morphism in $\I(X)$ is specified by
$$
\psi_k(\vec\alpha)=\alpha_{\sigma^{-1}(1)}\sqcup\dots\sqcup\alpha_{\sigma^{-1}(k)}.
$$
Since the classifying space functor preserves products, these functors give rise to maps
$$
\psi_k\co\mathcal E(k)\times \mathcal C(k)\times B\I(X)^k\to B\I(X),
$$
and it is straightforward to check that this defines an $\mathcal E\times \mathcal C$-action on $B\I(X)$. The associativity, unitality, and equivariance conditions may all be checked on the categorical level.
\end{proof}
Letting $\mathcal C$ be the commutativity operad $*$, it follows in particular that if $X$ is a commutative $\I$-space monoid, then $X_{h\I}$ inherits an $\mathcal E$-action. In this case the action is induced by a permutative structure on the category $\I(X)$ introduced in the above proof, cf.\ \cite{Ma1}, Section 4. This applies in particular to the $\I$-space $BF$ giving an 
$\mathcal E$-action on $BF_{h\I}$. We say that an operad $\mathcal C$ is augmented over 
$\mathcal E$ if there is a specified morphism of operads $\mathcal C\to \mathcal E$. In this case we may restrict an $(\mathcal E\times \mathcal C)$-action to the diagonal $\mathcal C$-action via the morphism $\mathcal C\to\mathcal E\times \mathcal C$.  
\begin{corollary}\label{hocolimUoperad}
If $\mathcal C$ is augmented over $\mathcal E$, then $\hocolim_{\mathcal I}$ induces a functor
$$
\hocolim_{\I}\co \I\mathcal U[\mathcal C]/BF\to \mathcal U[\mathcal C]/BF_{h\I}.
\eqno\qed
$$
\end{corollary}

\subsection{Operad actions preserved by $R$}
In order to prove that the $\I$-space lifting functor $R$ preserves operad actions, we need the following lemma in which we view $BF_{h\I}$ as a constant $\mathcal E$-$\mathcal I$-space.
\begin{lemma}\label{overlineBFaction}
The $\I$-space $\overline{BF}$ has an $\mathcal E$-action such that $\pi\co \overline{BF}\to BF_{h\I}$ is a morphism of $\mathcal E$-$\mathcal I$-spaces. 
\end{lemma} 
\begin{proof}
Consider more generally a commutative $\I$-space monoid $X$, and let $\overline X$ be 
the $\I$-space defined in Section \ref{invariancesection}. For each object $\mathbf n$ in $\I$, let $\I/\mathbf n(X)$ be the topological category whose classifying space is $\overline X(n)$. Thus, the object space is given by
$$
\coprod_{\alpha\co\mathbf m\to\mathbf n}X(m),
$$
where the coproduct is over the objects in $(\I\downarrow\mathbf n)$; see Appendix \ref{hocolimsection} for details. Consider for each $k$ the functor 
$$
\psi_k\co\tilde\Sigma_k\times \I/\mathbf n_1(X)\times\dots\times \I/\mathbf n_k(X)\to \I/
(\mathbf n_1\sqcup\dots\sqcup\mathbf n_k)(X)
$$
that maps a tuple of objects $\sigma$ in $\tilde\Sigma_k$ and $(\alpha_i,x_i)$ in $\I/\mathbf n_i(X)$ for $i=1,\dots,k$, to the object
$$
(\alpha,\mathbf x_{\sigma^{-1}(1)}\dots\mathbf x_{\sigma^{-1}(k)}),
$$ 
where $\alpha$ is the morphism
\begin{align*}
\alpha\co\mathbf m_{\sigma^{-1}(1)}\sqcup\dots\sqcup\mathbf m_{\sigma^{-1}(k)}
&\xr{\alpha_{\sigma^{-1}(1)}\sqcup\dots\sqcup\alpha_{\sigma^{-1}(k)}}\mathbf n_{\sigma^{-1}(1)}
\sqcup\dots\sqcup\mathbf n_{\sigma^{-1}(k)}\\
&\xr{\sigma^{-1}(n_{\sigma^{-1}(1)},\dots,n_{\sigma^{-1}(k)})}\mathbf n_{1}\sqcup\dots\sqcup\mathbf n_{k}
\end{align*}
and the second factor is the product of the elements $\mathbf x_{\sigma^{-1}(1)},\dots,
\mathbf x_{\sigma^{-1}(k)}$ using the monoid structure. The induced maps of classifying spaces
$$
\psi_k\co\mathcal E(k)\times\overline X(n_1)\times\dots\times\overline X(n_k)\to
\overline X(n_1+\dots+n_k)
$$
then specify the required $\mathcal E$-action on $\overline X$. With this definition it is clear that the canonical morphism $\overline X\to X_{h\I}$ is a morphism of $\mathcal E$-$\I$-spaces.
\end{proof} 

\begin{proposition}\label{operadrectificationprop}
Let $\mathcal C$ be an operad augmented over the Barratt-Eccles operad. Then the $\I$-space lifting functor $R$ induces a functor
$$
R\co\mathcal U[\mathcal C]/BF_{h\I}\to \I\mathcal U[\mathcal C]/BF.
$$
\end{proposition}
\begin{proof}
We give $BF_{h\I}$ the $\mathcal C$-action defined by the augmentation to $\mathcal E$. Let 
$f\co X\to BF_{h\I}$ be a map of $\mathcal C$-spaces and consider the diagram 
$$
X\xr{f} BF_{h\I}\xl{\pi} \overline{BF}
$$
defining $R_f(X)$. Pulling the $\mathcal E$-action on $\overline{BF}$ defined in Lemma 
\ref{overlineBFaction} back to a $\mathcal C$-action, this is a diagram of $\mathcal C$-$\I$-spaces. Thus, $R_f(X)$ is a homotopy pullback of a diagram in $\I\mathcal U[\mathcal C]$, hence is itself an object in this category and the projections $R_f(X)\to X$ and 
$R_f(X)\to \overline{BF}$ are maps of $\mathcal C$-$\mathcal I$-spaces, see \cite{Ma}, Section 1. Since the equivalence 
$t\co \overline{BF}\to BF$ is also a map of $\mathcal C$-$\I$-spaces, the conclusion follows. 
\end{proof}
Combining this with Corollary \ref{thomoperadcorollary} we get the following.
\begin{corollary}\label{TUoperad}
If $\mathcal C$ is an operad that is augmented over $\mathcal E$, then the Thom spectrum functor on $\mathcal U/BF_{h\I}$ induces a functor
$
T\co\mathcal U[\mathcal C]/BF_{h\I}\to\Sp^{\Sigma}[\mathcal C].
$\qed
 \end{corollary}

\section{The Thom isomorphism}\label{Thomisosection}
Let $MF$ be the symmetric Thom spectrum associated to the identity $BF\to BF$, and let 
$MSF$ be the symmetric Thom spectrum associated to the inclusion $BSF\to BF$. Here 
$SF(n)$ denotes the submonoid of orientation preserving based homotopy equivalences 
(those that are homotopic to the identity) and $BSF$ is the corresponding commutative 
$\I$-space monoid. 
 We first construct canonical orientations of these Thom spectra, and for this we need convenient models of Eilenberg-Mac Lane spectra. 
\subsection{Eilenberg-Mac Lane spectra and orientations}\label{orientationsection}
Let $A$ be a discrete ring, and write $A[-]$ for the functor that to a topological space $X$ associates the free topological $A$-module $A[X]$ generated by $X$, see e.g. \cite{W}, Section 2.3. In the special case where $X$ is the realization of a simplicial set $X_{\bullet}$ this may be identified with the realization of the simplicial $A$-module $A[X_{\bullet}]$. If $X$ is based, we write 
$A(X)$ for the topological $A$-module $A[X]/A[*]$. The functor $A(-)$ defined in this way is left adjoint to the forgetful functor from topological $A$-modules to based spaces.  It is well-known that $A(S^n)$ is a model of the Eilenberg-Mac Lane space $K(A,n)$, and that when equipped with the obvious structure maps this defines a model of the Eilenberg-Mac Lane spectrum for 
$A$ as a symmetric ring spectrum. In order to define the orientations, we shall consider a variant of this construction. Let $F_A(n)$ be the topological monoid of continuous $A$-linear endomorphisms of $A(S^n)$ and notice that by the above remarks this is homotopy equivalent to $A$ considered as a discrete multiplicative monoid. Writing $SF_A(n)$ for the connected component corresponding to the unit of $A$, this is then a contractible topological monoid. Applying the bar construction as in Section \ref{preliminariessection}, we get a well-based quasifibration
$$
B(*,SF_A(n),A(S^n))\to BSF_A(n),
$$
and we define the Eilenberg-Mac Lane spectrum $HA$ to be the symmetric spectrum with $n$th space 
$$
HA(n)=B(*,SF_A(n),A(S^n))/BSF_A(n).
$$
It is easy  to check that this is a commutative symmetric ring spectrum which is level-wise equivalent to the usual model for the Eilenberg-Mac Lane spectrum considered above. 
Since $HA$ is \emph{flat} in the sense of \cite{BCS}, the functor 
$HA\wedge(-)$ preserves stable equivalences between well-based spectra; this follows from a slight refinement of the argument used in  \cite{BCS}. (Alternatively, one can check that the arguments in \cite{Schw}, Proposition 5.14, works equally well with Quillen cofibrations of spaces replaced by our notion of ($h$-)cofibrations.) 
Let now  $A=\mathbb Z/2$ and observe that the functor $\mathbb Z/2(-)$ defines a map of sectioned quasifibrations
$$
B(*,F(n),S^n)\to B(*,SF_{\mathbb Z/2}(n),\mathbb Z/2(n)).
$$
The canonical orientation of $MF$ is the induced map of commutative symmetric ring spectra $MF\to H\mathbb Z/2$. Similarly, the functor $\mathbb Z(-)$ defines a map
of sectioned quasifibrations
$$
B(*,SF(n),S^n)\to B(*,SF_{\mathbb Z}(n),\mathbb Z(S^n))
$$
and the canonical orientation of $MSF$ is the induced map of commutative symmetric ring spectra $MSF\to H\mathbb Z$. 

\subsection{The Thom isomorphism}
We first consider the Thom isomorphism with $\mathbb Z/2$-coefficients. Given a map 
$f\co X\to BF_{h\I}$, the $\I$-space lift $R_f(X)\to BF$ induces a map of symmetric spectra $T(f)\to MF$, and we define the $H\mathbb Z/2$-orientation of $T(f)$ to be the composition
$$
T(f)\to MF\to H\mathbb Z/2.
$$
As explained in Section \ref{introThomiso}, the orientation induces a map of symmetric spectra
\begin{equation}\label{Z/2Thomequation}
T(f)\wedge H\mathbb Z/2\to X_+\wedge H\mathbb Z/2.
\end{equation}
Since our construction of the Thom spectrum functor has good properties both formally and homotopically, the proof that this is a stable equivalence is almost completely formal. 
\begin{theorem}
The map of symmetric spectra (\ref{Z/2Thomequation}) is a stable equivalence.
\end{theorem}
\begin{proof}
Both functor in the theorem are homotopy functors on $\mathcal U/BF_{h\I}$ in the sense that 
they take weak homotopy equivalences to stable equivalences; this follows from Corollary 
\ref{TRcorollary} and the fact that $H\mathbb Z/2$ is flat. Thus, we may assume that $X$ is a CW-complex and consider the filtration of $X$ by skeleta $X^{n}$ such that $X^{-1}$ is the empty set and $X^n$ is homeomorphic to the pushout of a diagram of the form
$$
X^{n-1}\leftarrow\coprod S^{n-1}\to \coprod D^n.
$$  
Since both functors in the theorem preserve pushouts and $h$-cofibrations, it suffices by 
\cite{MMSS}, Theorem 8.12, to consider the case where the domain of $f$ is of the form $D^n$ or $S^n$. If $f$ is the inclusion of the basepoint $*\to BF_{h\I}$, then the unit of the $\I$-space monoid $R_f(*)$ gives a stable equivalence $S\to T(f)$ and the composition
$$
S\wedge H\mathbb Z/2\xr{\sim}T(f)\wedge H\mathbb Z/2\to *_+\wedge H\mathbb Z/2
$$
is the identity on $H\mathbb Z/2$. Using the homotopy invariance of the Thom spectrum functor, this easily implies the result for $D^n$. Identifying $S^n$ with the pushout of the diagram 
$D^n\leftarrow S^{n-1}\to D^n$,
the result for $S^n$ then follows by an inductive argument.
\end{proof}
 
\subsection{The integral Thom isomorphism} 
Using the commutative $\I$-space monoid $BSF$ instead of $BF$,  we get a a monoidal $\I$-space 
lifting  functor
$$
R\co \mathcal U/BSF_{h\I}\to \I\mathcal U/BSF
$$ 
defined in analogy with the $\I$-space lifting functor on $\mathcal U/BF_{h\I}$. The two lifting functors are related by a diagram
$$
\begin{CD}
\mathcal U/BSF_{h\I}@>R>> \I\mathcal U/BSF\\
@VVV @VVV\\
\mathcal U/BF_{h\I}@>R>> \I\mathcal U/BF,
\end{CD}
$$
which is commutative up to natural $\I$-equivalence. Thus, the two natural ways to define a Thom spectrum functor on $\mathcal U/BSF_{h\I}$ are equivalent up to stable equivalence. For the definition of orientations it is most convenient to define the Thom spectrum functor on 
$\mathcal U/BSF_{h\I}$ to be the composition 
$$
T\co \mathcal U/BSF_{h\I}\xr{R}\I\mathcal U/BSF\to \I\mathcal U/BF\xr{T}\Sp^{\Sigma}.
$$
With this definition we have a canonical integral orientation of the Thom spectrum associated to a map $X\to BSF_{h\I}$, defined by the composition
$
T(f)\to MSF\to H\mathbb Z.
$
The orientation again gives rise to a map of symmetric spectra
\begin{equation}\label{ZThomequation}
T(f)\wedge H\mathbb Z\to X_+\wedge H\mathbb Z
\end{equation}
and the proof of the integral version of the Thom isomorphism theorem is completely analogous to the $H\mathbb Z/2$-version.
\begin{theorem}
The map (\ref{ZThomequation}) is a stable equivalence.\qed
\end{theorem}

We can now verify the claim in Theorem \ref{structuredThomiso} that the Thom equivalence is strictly multiplicative.

\medskip
\begin{proof}[Proof of Theorem \ref{structuredThomiso}]
Let $H$ denote either one of the commutative symmetric ring spectra $H\mathbb Z/2$ or 
$H\mathbb Z$, and view $H$ as an object in $\Sp^{\Sigma}[\mathcal C]$ by 
projecting $\mathcal C$ onto the commutativity operad. We claim that 
(\ref{Z/2Thomequivalence}) is a diagram in $\Sp^{\Sigma}[\mathcal C]$ when we give each of the terms the diagonal $\mathcal C$-action.  For the first two maps this follows from Proposition 
\ref{operadrectificationprop} and Corollary \ref{TUoperad}, which imply that the Thom diagonal and $T(f)\to MF$ (or $T(f)\to MSF$) are both 
$\mathcal C$-maps. For the last map the claim follows from the fact that the multiplication $H\wedge H\to H$ is a map of commutative symmetric ring spectra, hence in particular a 
$\mathcal C$-map.
\end{proof}

\section{Symmetrization of diagram Thom spectra}\label{diagramThomsection}
In this section we first generalize the definition of the symmetric Thom spectrum functor to other types of diagram spectra. We then show how the results in the previous sections can be used to turn such diagram Thom spectra into symmetric spectra.

\subsection{Diagram spaces and diagram Thom spectra} 
Given a small category $\mathcal D$, we define a \emph{$D$-space} to be a functor 
$X\co\mathcal D\to \mathcal U$ and we write $\mathcal D\mathcal U$ for the category of such functors. Suppose that we are given a functor $\phi\co \mathcal D\to\mathcal I$. 
Then we can generalize the notion of a symmetric spectrum by introducing the topological category $\mathcal D_S$ that has the same objects as $\mathcal D$, but whose morphism spaces are defined by
$$
\mathcal D_S(a,b)=\bigvee_{\alpha\in \mathcal D(a, b)}S^{b-\alpha},
$$
where $S^{b-\alpha}$ is shorthand notation for $S^{\phi(b)-\phi(\alpha)}$, cf.\ Section 
\ref{symmetricspectrumsection}. The composition is defined as for $\I_S$. We define a \emph{$\mathcal D$-spectrum} to be a continuous based functor 
$\mathcal D_S\to \mathcal T$ and we write $\mathcal D_S\mathcal T$ for the category of such functors. Thus, a $\mathcal D$-spectrum is given by a family of based spaces $X(a)$ indexed by the objects $a$ in $\mathcal D$, together with a family of based structure maps 
$S^{b-\alpha}\wedge X(a)\to X(b)$ indexed by the morphisms $\alpha\co a\to b$ in $\mathcal D$. It is required (i) that the structure map associated to an identity morphism $1_a\co a\to a$ is the canonical identification $S^0\wedge X(a)\to X(a)$, and (ii) that  given a pair of composable morphisms $\alpha\co a\to b$ and $\beta\co b\to c$, the following diagram is commutative
\[
\begin{CD}
S^{c-\beta}\wedge S^{b-\alpha}\wedge X(a)@>>> S^{c-\beta}\wedge X(b)\\
@VVV @VVV\\
S^{c-\beta\alpha}\wedge X(a)@>>> X(c).
\end{CD}
\]
In particular, if $\phi$ denotes the identity functor on $\I$, then  $\mathcal I_S\mathcal T$ is an alternative notation for the category of symmetric spectra. Suppose now that $\mathcal D$ has the structure of a strict monoidal category. As in the case of $\I$-spaces, 
$\mathcal D\mathcal U$ inherits a monoidal structure from $\mathcal D$ which is symmetric monoidal if $\mathcal D$ is. If in addition $\phi$ is strict monoidal, then the monoidal structure of $\mathcal D$ also induces a monoidal structure on $\mathcal D_S$ which is symmetric monoidal if $\mathcal D$ and $\phi$ are. This in turn induces a monoidal structure on the category of $\mathcal D$-spectra $\mathcal D_S\mathcal T$ which again is symmetric monoidal if 
$\mathcal D$ and  $\phi$ are. 
The $\I$-space $BF$ pulls back to a $\mathcal D$-space via $\phi$ and the definition of the symmetric Thom spectrum functor immediately generalizes to give a Thom spectrum functor
$$
T\co \mathcal D\mathcal U/BF\to\mathcal D_S\mathcal T.
$$
The proof of Theorem \ref{monoidaltheorem} generalizes to show that this is a strong monoidal functor which is symmetric monoidal if $\mathcal D$ and $\phi$ are.
  
\subsection{Examples of diagram Thom spectra}\label{examplesection}
Many examples of Thom spectra arise from compatible families of groups over the topological monoids $F(n)$. It often happens that such a family defines a $\mathcal D$-diagram of groups for some strict monoidal category $\mathcal D$ over $\I$ and if the induced maps of classifying spaces define a 
$\mathcal D$-space over $BF$ we get an associated $\mathcal D$-Thom spectrum. We begin by fixing notation for some of the relevant categories. For each $k\geq 1$ we have the strict symmetric monoidal faithful functor
$$
\psi_k\co \I\to\I,\quad \mathbf n\mapsto \underbrace{\mathbf k\sqcup\dots\sqcup\mathbf k}_{n}.
$$ 
and we write $\I[k]$ for its image in $\I$. Thus, $\I[k]$ is a strict symmetric monoidal category whose objects have cardinality a multiple of $k$, and whose morphisms permute blocks of $k$ letters simultaneously. Let us write $\mathcal M$ for the subcategory of injective order preserving morphisms in $\I$. This inherits a strict monoidal (but not symmetric monoidal) structure from 
$\I$ and we similarly define monoidal subcategories $\mathcal M[k]$ for $k\geq 1$. 

\medskip
\begin{example}[The classical groups]
The orthogonal groups $O(n)$ and the special orthogonal groups $SO(n)$ define the commutative  $\I$-space monoids $BO$ and $BSO$ that give rise to the commutative symmetric ring spectra $MO$ and $MSO$. The unitary groups $U(n)$ and the special unitary groups $SU(n)$ define the commutative $\I[2]$-space monoids $BU$ and $BSU$ that give rise to the commutative $\I[2]$-ring spectra $MU$ and $MSU$. The symplectic groups $Sp(n)$ define the commutative 
$\I[4]$-space monoid $BSp$ that gives rise to the commutative $\I[4]$-spectrum $MSp$. 
\end{example}

\begin{example}[Discrete groups and $\I$-spaces]
The symmetric groups $\Sigma_n$ define the commutative $\I$-space monoid $B\Sigma$ in which the monoid structure is induced by concatenation of permutations. This gives rise to the commutative symmetric ring spectrum $M\Sigma$ whose associated bordism theory has been studied by Bullett \cite{Bu}. Other systems of discrete groups that give rise to symmetric ring spectra include the general linear groups $GL_n(\mathbb Z)$, the groups $(\mathbb Z/2)^n$ of diagonal matrices with entries $\pm 1$, and the groups $\Sigma_n\wr \mathbb Z/2$ of permutation matrices with entries $\pm 1$. For details and more examples, see \cite{Bu} and \cite{CV}.
\end{example}

\begin{example}[Braid groups and $\mathcal M$-spaces]
The family of braid groups $\mathfrak B(n)$ defines an $\mathcal M$-space monoid 
$B\mathfrak B$ in a natural way. We refer to \cite{Bi} for the definition and basic properties of the braid groups. If we view an element of $\mathfrak B(n)$ as a system of $n$ strings in the usual way, then the monoid structure on $B\mathfrak B$ is induced by concatenation of such systems.
Let $\rho$ denote the sequence of monoid homomorphisms 
\[
\rho_n\co\mathfrak B(n)\to \Sigma_n\to F(n)
\]
where the first map takes a system of strings to the induced permutation of the endpoints and the second map in the canonical inclusion. This defines a map $B\rho\co B\mathfrak B\to BF$ of $\mathcal M$-space monoids and we write $M\mathfrak B$ for the associated 
$\mathcal M$-Thom ring spectrum. The underlying spectrum of 
$M\mathfrak B$ has been analyzed in \cite{Bu} and \cite{Co}, where it is shown to be equivalent to the Eilenberg-Mac Lane spectrum $H\mathbb Z/2$. Suppose that $G$ is an $\mathcal M$-diagram of groups over the monoids $F(n)$ and that the homomorphisms $\rho_n$ can be factored as $\mathfrak B(n)\to G(n)\to F(n)$. If the $\mathcal M$-space $BG$ admits a monoid structure such that the induced map $B\mathfrak B\to BG$ is a map of $\mathcal M$-space monoids over $BF$ it then follows that $MG$ is an $\mathcal M$-module spectrum over $M\mathfrak B$. For example, this applies to $M\Sigma$ and $MO$ but not to $MSO$. 
We show how to symmetrize the constructions so as to get symmetric Thom spectra in Section 
\ref{symmetrizationsubsection}.
Again we refer to \cite{Bu}, \cite{Co},  \cite{CV} for further examples.
\end{example}

\begin{example}[Maps to $BF(k)$ and $\I{[}k{]}$-spaces]
For our next class of examples we need some preliminary definitions. Let $X$ be a based space and let $X^{\bullet}$ be the $\mathcal I$-space defined by $\mathbf n\mapsto X^n$. Given a morphism $\alpha\co\mathbf m\to\mathbf n$, the induced map $\alpha_*\co X^m\to X^n$ is defined by
$$
\alpha_*(x_1,\dots,x_m)=(x_{\alpha^{-1}(1)},\dots,x_{\alpha^{-1}(n)}),
$$ 
with the convention that $x_{\emptyset}$ is the base point in $X$. We give $X^{\bullet}$ the structure of a commutative $\I$-space monoid using the identifications $X^m\times X^n=X^{m+n}$.  Suppose now that $f\co X\to BF(k)$ is a based map. Then we view $X^{\bullet}$ as an $\I[k]$-space via the isomorphism $\psi_k\co\I\to\I[k]$, and  the maps 
$$
X^n\to BF(k)^n\xr{\mu}BF(\underbrace{\mathbf k\sqcup\dots\sqcup\mathbf k}_{n})
$$
define a map of $\I[k]$-space monoids $X^{\bullet}\to BF$, where we view $BF$ as an $\I[k]$-space by restriction. We write $MX^{\wedge\infty}$ for the associated commutative $\I[k]$-ring spectrum, the function $f$ being understood. In the cases $X=BO(1)$ and $X=BU(1)$, we get the Thom spectra  $MO(1)^{\wedge\infty}$ and $MU(1)^{\wedge\infty}$ that represent the bordism  theories of manifolds with stable normal bundle given as an ordered sum of real or complex line bundles. These Thom spectra have been analyzed by Arthan and Bullett \cite{AB}, \cite{Bu}. Letting $X=BO(k)$ or $X=BU(k)$, we similarly get the $\I[k]$-spectrum $MO(k)^{\wedge\infty}$ and 
the $\I[2k]$-spectrum $MU(k)^{\wedge\infty}$.
\end{example}

\subsection{Symmetrization of diagram Thom spectra}\label{symmetrizationsubsection}
As demonstrated in the last section, many Thom spectra naturally arise as $\mathcal D$-Thom spectra associated to maps of $\mathcal D$-spaces $f\co X\to BF$ for suitable monoidal 
categories $\mathcal D$ over $\I$. In the applications it is often convenient to replace such a 
$\mathcal D$-Thom spectrum by a symmetric Thom spectrum and our preferred way of doing is to first transform $f$ to a map of $\I$-spaces and then evaluate the symmetric Thom spectrum functor on this transformed map. We shall discuss two ways of performing this ``symmetrization'' procedure: in this section we consider symmetrizations using the $\I$-space lifting functor $R$ and in the next section we consider symmetrizations via (homotopy) Kan extension. 

For simplicity we shall from now on assume that $\mathcal D$ is a monoidal subcategory of $\I$ such that the intersection $\mathcal D\cap\mathcal N$ is a cofinal subcategory of $\mathcal N$. 
Thus, any $\mathcal D $-spectrum has an underlying $\mathcal D\cap\mathcal N$-spectrum and we define the spectrum homotopy groups in the usual way by evaluating the colimit of the associated $\mathcal D\cap \mathcal N$-diagram of homotopy groups. Given a map 
$f\co X\to BF$, we write (by abuse of notation) $f_{h\mathcal D}$ for the composite map
\[
f_{h\mathcal D}\co X_{h\mathcal D}\xr{f_{h\mathcal D}} BF_{h\mathcal D}\to BF_{h\I}.
\] 
Applying the $\I$-space lifting functor $R$ to this map we get a functor 
\begin{equation}\label{RDlifting}
\mathcal D\mathcal U/BF\to \I\mathcal U/BF,\qquad (X\xr{f} BF)
\mapsto (R_{f_{h\mathcal D}}(X_{h\mathcal D})\xr{R(f_{h\mathcal D})} BF).
\end{equation}
We say that a map of $\mathcal D$-spaces $X\to Y$ is a \emph{$\mathcal D$-equivalence}
if the induced map $X_{h\mathcal D}\to Y_{h\mathcal D}$ is a weak homotopy equivalence.

\begin{lemma}
The restriction of $R_{f_{h\mathcal D}}(X_{h\mathcal D})$ to a $\mathcal D$-space is related to $X$ by a chain of $\mathcal D$-equivalences over $BF$.
\end{lemma}
\begin{proof}
In analogy with the case of $\I$-spaces considered in Section \ref{invariancesection}, there is a diagram of $\mathcal D$-equivalences 
$
X\leftarrow \overline X\to R_{f_{h\mathcal D}}(X_{h\mathcal D})
$ 
over $BF$.
\end{proof}
It follows from the $\mathcal D$-space analogue of B\"okstedt's approximation Lemma \ref{Blemma} that if $X\to Y$ is a $\mathcal D$-equivalence of convergent $\mathcal D$-spaces $X$ and $Y$, then the connectivity of the maps $X(d)\to Y(d)$ tends to infinity with $d$. The previous lemma therefore has the following consequence.

\begin{proposition}\label{D-restriction-equivalence}
If $X$ is a convergent $\mathcal D$-space and $f\co X\to BF$ is a map of $\mathcal D$-spaces which is level-wise $T$-good, then the restriction of the symmetric spectrum $T(R(f_{h\mathcal D}))$ to a $\mathcal D$-spectrum is 
$\pi_*$-equivalent to $T(f)$.\qed
\end{proposition}
This construction preserves multiplicative structures in the sense that if $X$ is a $\mathcal D$-space monoid and $f\co X\to BF$ a map of $\mathcal D$-space monoids, then $f_{h\mathcal D}$ is a map of topological monoids and $T(f_{h\mathcal D})$ is a symmetric ring spectra by Lemma  
\ref{monoid-ring}. In the following we consider the effect of applying the construction to the examples considered in the previous section. 

\begin{example}\label{I[k]-example}
Let $X$ be an $\I[k]$-space with an action of an operad $\mathcal C$ that is augmented over the Barratt-Eccles operad. If $f\co X\to BF$ is a map of $\mathcal C$-$\I[k]$-spaces, then the induced map
$$
f_{h\I[k]}\co X_{h\I[k]}\to BF_{h\I[k]}\to BF_{h\I}
$$
is map of $\mathcal C$-spaces and it follows from Corollary \ref{TUoperad} that the symmetric spectrum $T(f_{h\I[k]})$ inherits a $\mathcal C$-action. Here we use the canonical isomorphism of categories $\psi_k\co \I\to\I[k]$ to identify $X_{h\I[k]}$ with $(\psi_k^*X)_{h\I}$, and we transfer the $\mathcal C$-action on $(\psi_k^*X)_{h\I}$ defined in Corollary \ref{hocolimUoperad} to 
$X_{h\I[k]}$ via this identification. This applies in particular to the map of commutative 
$\I[2]$-spaces $BU\to BF$ to give a model of the Thom spectrum $MU$ as a symmetric ring spectrum with an action of the Barratt-Eccles operad. We shall see how to realize $MU$ as a strictly commutative symmetric ring spectrum in Example \ref{MUexample}.
\end{example}

\begin{example}
Let as before $B\mathfrak B$ denote the $\mathcal M$-space monoid defined by the braid groups. We shall identify the map $B\mathfrak B_{h\mathcal N}\to B\mathfrak B_{h\mathcal M}$ in terms of Quillen's plus construction. Firstly, it follows from the homological stability of the braid groups (see \cite{CLM}, III, Appendix) and the homological version of Lemma \ref{Blemma}, that this map is a homology isomorphism. Secondly, the monoidal structure of 
$\mathcal M$ gives 
$B\mathfrak B_{h\mathcal M}$ the structure of a topological monoid, which in particular implies that its fundamental group is abelian. Thus, the map in question has the effect of abelianizing the fundamental group. The space $B\mathfrak B_{h\mathcal N}$ may be identified with the classifying space of the infinite braid group $\mathfrak B(\infty)$. Since the commutator subgroup of the latter is perfect, it follows from the above that we may identify 
$B\mathfrak B_{h\mathcal M}$ with Quillen's plus construction $B\mathfrak B_{h\mathcal N}^+$. 
It is proved in \cite{Co} that there is a homotopy commutative diagram
$$
\begin{CD}
B\mathfrak B_{h\mathcal N}@>\theta >>\Omega^2(S^3)\\
@VV(B\rho)_{h\mathcal N}V @VV\eta V\\
BF_{h\mathcal N}@=BF_{h\mathcal N},
\end{CD}
$$
where $\eta$ denotes the ``Mahowald orientation'', that is, the extension of the non-trivial map $S^1\to BF_{h\mathcal N}$ to a 2-fold loop map. It is a theorem of Mahowald  \cite{Mah1}, that the Thom spectrum of $\eta$ is stably equivalent to the Eilenberg-Mac Lane spectrum $H\mathbb Z/2$. By the universal property of the plus construction, we conclude from the above that there is a homotopy commutative diagram
$$
\begin{CD}
B\mathfrak B_{h\mathcal M}@>\sim >> \Omega^2(S^3)\\
@VV(B\rho)_{h\mathcal M} V @VV \eta V\\
BF_{h\I}@= BF_{h\I},
\end{CD}
$$
where the upper map is a homotopy equivalence as indicated. Consequently, the symmetric ring spectrum $T((B\rho)_{h\mathcal M})$ is a model of $H\mathbb Z/2$. 
\end{example} 

\begin{example}
Let $BGL(\mathbb Z)$ be the commutative $\I$-space monoid associated to the general linear groups $GL_n(\mathbb Z)$. As in the case of the braid groups, we may identify $BGL(\mathbb Z)_{h\mathcal N}\to BGL(\mathbb Z)_{h\I}$ in terms of Quillen's plus construction. Indeed, by the homological stability of the groups $GL_n(\mathbb Z)$, it follows that this map is a homology equivalence. Since $BGL(\mathbb Z)_{h\I}$ is a topological monoid it has abelian fundamental group, hence it may be identified with $BGL_{\infty}(\mathbb Z)^+$; the base point component of Quillen's algebraic K-theory space. In similar fashion, starting with the $\I$-space 
$B\Sigma$, we may identify $B\Sigma_{h\I}$ with 
$B\Sigma_{\infty}^+$, which by the Barratt-Priddy-Quillen Theorem is equivalent to the base point component of $Q(S^0)$. 
\end{example}

\begin{example}\label{Xbulletexample}
Let $f\co X\to BF(k)$ be  a based map and consider the associated map of $\I[k]$-spaces $X^{\bullet}\to BF$. It is proved in \cite{Sch2} that if $X$ is well-based and connected, then 
$X^{\bullet}_{h\I}$ is a model of the infinite loop space $Q(X)$. Identifying $X^{\bullet}_{h\I}$ with $X^{\bullet}_{h\I[k]}$ via the isomorphism $\psi_k$, it follows as in Example \ref{I[k]-example} that the induced map $X^{\bullet}_{h\I[k]}\to BF_{h\I}$ is a map of 
$\mathcal E$-spaces which models the usual extension of $f$ to a map of infinite loop spaces. 
If we instead think of $X^{\bullet}$ as an $\mathcal M$-space by restriction, then one can show that $X_{h\mathcal M}$ is homotopy equivalent to the colimit $X^{\bullet}_{\mathcal M}$, that is, to the free based monoid generated by $X$. By a theorem of James \cite{Ja} 
the latter is a model of 
$\Omega\Sigma(X)$, and the map $X^{\bullet}_{h\mathcal M}\to X^{\bullet}_{h\mathcal I}$ corresponds to the inclusion of $\Omega\Sigma(X)$ in $Q(X)$.   
\end{example}

\begin{example}
Let $\mathcal E_k$ be the $k$th stage of the Smith filtration of the Barratt-Eccles operad 
$\mathcal E$ and write $X\mapsto E_k(X)$ for the associated monad on based spaces, see \cite{Ber}, \cite{Sm}. Thus, $\mathcal E_k$ is equivalent to the little $k$-cubes operad, and if 
$X$ is a well-based connected space, then $E_k(X)$ is a combinatorial model of 
$\Omega^k\Sigma^k(X)$. Given a based map $f\co X\to BF_{h\I}$, we use that 
$\mathcal E_k$ is augmented over $\mathcal E$ to extend $f$ to a map of $\mathcal E_k$-spaces
$$
E_k(f)\co E_k(X)\to E_k(BF_{h\I})\to BF_{h\I},
$$
which for connected $X$ is a model of the usual extension of $f$ to a $k$-fold loop map. It follows that the associated symmetric Thom spectrum $T(E_k(f))$ is equipped with 
an $\mathcal E_k$-action.
\end{example}

\subsection{Symmetrization via Kan Extension}
Let again $\mathcal D$ be a monoidal subcategory of $\I$ such that 
$\mathcal D\cap \mathcal N$ is cofinal in $\mathcal N$ and let us write $j\co \mathcal D\to \I$ for the inclusion. We first consider homotopy Kan extensions along $j$. Recall that given a 
$\mathcal D$-space $X$, the homotopy Kan extension is the $\I$-space $j_*^h(X)$ defined by
\[
j_*^h(X)(n)=\hocolim_{(j\downarrow \mathbf n)}X\circ \pi_n,
\]
see Appendix \ref{homotopyKansection}. The functor $j_*^h(-)$ induces a functor
\[
j^h_*\co \mathcal D\U/BF\to \I\U/BF,\quad (f\co X\to BF)
\mapsto (j_*^h(f)\co j_*^h(X)\to j_*^h(BF)\to BF)
\]
which is $\I$-equivalent to the functor (\ref{RDlifting}) in the sense of the following lemma.
\begin{lemma}\label{j-Kan-lemma}
There is a natural $\I$-equivalence $j_*^h(X)\to R_{f_{h\mathcal D}}(X_{h\mathcal D})$ of $\I$-spaces over $BF$.
\end{lemma}
\begin{proof}
Notice first that there is a commutative diagram
\[
\begin{CD}
j_*^h(X)@>>> \overline{BF}\\
@VVV @VVV\\
X_{h\mathcal D}@>>>BF_{h\I},
\end{CD}
\]
inducing a map of $\I$-spaces $j_*^h(X)\to R_{f_{h\mathcal D}}(X_{h\mathcal D})$ over $BF$. 
Since this is also a map over $X_{h\mathcal D}$, the result follows from 
the fact that $j_*^h(X)\to X_{h\mathcal D}$ and 
$R_{f_{h\mathcal D}}(X_{h\mathcal D})\to X_{h\mathcal D}$ are $\I$-equivalences, see Lemma \ref{homotopyKanequivalence} and the proof of Proposition \ref{RTgood}. 
\end{proof}
We now turn to (categorical) Kan extensions. Given a $\mathcal D$-space, the Kan extension $j_*(X)$ is defined as the homotopy Kan extension except that we use the colimit instead of the homotopy colimit, that is,
\[
j_*(X)(n)=\colim_{(j\downarrow \mathbf n)}X\circ \pi_{\mathbf n}.
\]
The functor $j_*$ is left adjoint to the functor that pulls an $\I$-space back to a $\mathcal D$-space via $j$ and it induces a functor 
\[
j_*\co \mathcal D\mathcal U/BF\to \I\mathcal U/BF,\qquad (X\to BF)\mapsto 
(j_*(X)\to j_*(BF)\to BF)
\]
where the map $j_*(BF)\to BF$ is the counit of the adjunction. This functor is strong monoidal and is symmetric monoidal if $\mathcal D$ and $j$ are. Thus, in the latter case it takes commutative $\mathcal D$-space monoids to commutative $\I$-space monoids. The drawback of using the categorical Kan extension is of course that it is homotopically well-behaved only under suitable cofibration conditions on the $\mathcal D$-space $X$ and the main purpose of this section is to formulate such conditions. More precisely, we shall consider an inclusion $j\co \mathcal D\to \I$ of a (not necessarily monoidal) subcategory $\mathcal D$ of $\I$ 
and we shall formulate conditions on $\mathcal D$ and $X$ which ensure that the canonical map $j_*^h(X)\to j_*(X)$ is a level-wise equivalence. 
Given an object $\mathbf d_0$ of $\mathcal D$, consider the  category $(\mathcal D\downarrow \mathbf d_0)$ of objects in 
$\mathcal D$ over $\mathbf d_0$ and let $\partial \mathbf d_0$ be the subcategory 
obtained by excluding the terminal objects.

\begin{lemma}\label{modelcriterion}
Let $j\co\mathcal D\to\I$ be the inclusion of a subcategory $\mathcal D$ and suppose that 
$X$ is a $\mathcal D$-space such that the map
\[
\colim_{\partial\mathbf d_0} X\circ \pi_{\mathbf d_0}\to 
\colim_{(\mathcal D\downarrow\mathbf d_0)}
X\circ\pi_{\mathbf d_0}=X( d_0)
\]
is a cofibration for all objects $\mathbf d_0$ in $\mathcal D$. Then
$j_*^h(X)\to j_*(X)$ is a level-wise equivalence. 
\end{lemma}
\begin{proof}
Notice first that the category $(j\downarrow\mathbf n)$ is a preorder in the sense that the morphism sets have at most one element. Choosing a representative for each isomorphism class we get an equivalent skeleton subcategory $\mathcal A(\mathbf n)$ (in fact a partially ordered set), and it suffices to show that the map
\[
\hocolim_{\mathcal A(\mathbf n)}X\circ \pi_{\mathbf n}\to\colim_{\mathcal A(\mathbf n)}
X\circ\pi_{\mathbf n} 
\] 
is a weak homotopy equivalence. The advantage of this is that the category 
$\mathcal A(\mathbf n)$ is \emph{very small} in the sense that its nerve only has finitely many non-degenerate simplices. In this situation there is a general model categorical criterion for comparing the homotopy colimit to the colimit, see \cite{DS}, Section 10. Working in the Str\o m model category on $\mathcal U$ \cite{Str2}, we must check that for each object $a$ in 
$\mathcal A(n)$ the map
$$
\colim_{\partial a}X\circ \pi_{\mathbf n}\circ\pi_a\to 
\colim_{(\mathcal A(n)\downarrow a)}X\circ\pi_{\mathbf n}\circ\pi_a
=X(\pi_{\mathbf n}(a))   
$$
is a cofibration. Here we use the notation $\partial a$ for the subcategory of 
$(\mathcal A(\mathbf n)\downarrow a)$ obtained by excluding the terminal object. It remains to see that if $a$ is an object of the form $\mathbf d_0\to\mathbf n$, then this criterion is the same as that stated in the lemma. On the one hand we may view $(\mathcal A(n)\downarrow a)$ as a skeleton subcategory of $((j\downarrow\mathbf n)\downarrow a)$ and on the other hand we may identify the latter category with $(\mathcal D\downarrow\mathbf d_0)$. Taken together this gives a homeomorphism
$$
\colim_{\partial a}X\circ\pi_{\mathbf n}\circ\pi_a\cong\colim_{\partial \mathbf d_0}X\circ\pi_{\mathbf d_0}
$$
and the conclusion follows.      
\end{proof}

The criterion in Lemma \ref{modelcriterion} is not very practical and in order to have a more convenient formulation we impose conditions on the subcategory $\mathcal D$ of $\I$. We say that $\mathcal D$ has the \emph{intersection property} if each diagram in $\mathcal D$ of the form
$$
\mathbf d_1\xr{\delta_1}\mathbf d_{12}\xl{\delta_2}\mathbf d_2
$$ 
can be completed to a commutative square
\begin{equation}\label{Dsquare}
\begin{CD}
\mathbf d_0@>>>\mathbf d_1\\
@VVV @VV\delta_1 V\\
\mathbf d_2@>\delta_2 >>\mathbf d_{12}
\end{CD}
\end{equation}
in $\mathcal D$ such that the image of the composite morphism equals the intersection of the images of $\delta_1$ and $\delta_2$. For example, 
the monoidal subcategories $\I[k]$ and $\mathcal J[k]$ have the intersection property for all $k\geq 1$. We say that a $\mathcal D$-space $X$ is \emph{intersection cofibrant} if for any diagram of the form (\ref{Dsquare}), such that the intersection of the images of $\delta_1$ and $\delta_2$ equals the image of the composite morphism, the induced map
$$
X(d_1)\cup_{X(d_0)}X(d_2)\to X(d_{12})
$$
is a cofibration. By Lillig's union theorem \cite{Li} for cofibrations, this is equivalent to the requirement that (i) any morphism $\mathbf d_1\to\mathbf d_2$ in $\mathcal D$ induces a cofibration $X(d_1)\to X(d_2)$, and (ii) that the intersection of the images of $X(d_1)$ and $X(d_2)$ in $X(d_{12})$ equals the image of $X(d_0)$.

\begin{proposition}\label{intersectionproposition}
Let $j\co \mathcal D\to\I$ be the inclusion of a subcategory $\mathcal D$ which has the intersection property and let $X$ be a $\mathcal D$-space which is intersection cofibrant. Then the map $j_*^h(X)\to j_*(X)$ is a level-wise equivalence. 
\end{proposition}
\begin{proof}
We show that the assumptions on $\mathcal D$ and $X$ imply that the criterion in Lemma \ref{modelcriterion} is satisfied.
Given an object $\mathbf d_0$ in $\mathcal D$, consider the \emph{range functor}
\[
r\co (\mathcal D\downarrow\mathbf d_0)\to(\I\downarrow\mathbf d_0)\to\mathcal P(\mathbf d_0),\quad r(\mathbf d\xr{\delta}\mathbf d_0)=\delta(\mathbf d)\subseteq \mathbf d_0, 
\]
where $\mathcal P(\mathbf d_0)$ denotes the category of subsets and inclusions in 
$\mathbf d_0$. The assumption that $\mathcal D$ has the intersection property implies that  
the image of $r$ is a full subcategory of $\mathcal P(\mathbf d_0)$ that is closed under inclusions and that $r$ defines an equivalence of categories between 
$(\mathcal D\downarrow\mathbf d_0)$ and its image. Thus, we might as well view 
$X\circ \pi_{\mathbf d_0}$ as a diagram $U\mapsto X(U)$ indexed on the objects $U$ in a full subcategory $\mathcal A$ of $\mathcal P(\mathbf d_0)$ that is closed under intersections. By assumption (i) above we may view $X(U)$ as a closed subspace of $X(d_0)$ for all $U\in \mathcal A$ and by assumption (ii) we have the equality
$$
X(U)\cap X(V)= X(U\cap V)
$$   
for all pairs of objects $U$ and $V$ in $\mathcal A$. It therefore follows from the gluing lemma for continuous functions on a union of closed subspaces that $\colim_{\partial\mathbf d_0}X\circ\pi_{\mathbf d_0}$ is homeomorphic to the union of the subspaces $X(U)$ of $X(d_0)$ for 
$U\neq \mathbf d_0$. The conclusion then follows from Lillig's union theorem for cofibrations \cite{Li}.
\end{proof}

\begin{example}\label{MUexample}
Let $j\co \I[2]\to\I$ be the inclusion of the symmetric monoidal subcategory $\I[2]$. Since $\I[2]$ has the intersection property and the commutative $\I[2]$-space monoid $BU$ is intersection cofibrant, it follows from Lemma \ref{j-Kan-lemma} and Proposition \ref{intersectionproposition}  that there is a chain of $\I$-equivalences 
\[
j_*(BU)\xl{\sim}j_*^h(BU)\xr{\sim}R(BU_{h\I})
\]
over $BF$. Thus, it follows from Proposition \ref{D-restriction-equivalence} together with Theorem \ref{introinvariance} and Lemma \ref{Top(n)-lemma} that
applying the symmetric Thom spectrum functor to the commutative $\I$-space monoid $j_*(BU)$ gives a commutative symmetric ring spectrum which is a model of $MU$.  
\end{example}

\subsection{Orthogonal Thom spectra and diagram lifting}\label{orthogonalsection}
Recall from \cite{MMSS} that an orthogonal spectrum is a spectrum $X$ such that the $n$th space $X(n)$ has an action of the orthogonal group $O(n)$, and such that the iterated structure maps 
$
S^m\wedge X(n)\to X(m+n)
$
are $O(m)\times O(n)$-equivariant. We write $\Sp^O$ for the category of orthogonal spectra. Let $\mathcal V$ be the topological category whose objects are the vector spaces $\mathbb R^n$, and whose morphisms are the linear isometries. A $\mathcal V$-space is a continuous functor 
$\mathcal V\to \mathcal U$, and we write $\mathcal V\mathcal U$ for the category of such functors. The symmetric monoidal structure of $\mathcal V$ induces a symmetric monoidal structure on $\mathcal V\mathcal U$ in the usual way, and the $\I$-space $BF$ extends to a commutative $\mathcal V$-space monoid. Applying the Thom space construction level-wise as in the definition of the symmetric Thom spectrum functor, we get the symmetric monoidal orthogonal Thom spectrum functor
$$
T\co \mathcal V\mathcal U/BF\to \Sp^{O}.
$$
In order to construct orthogonal Thom spectra from space level data, we need a $\mathcal V$-space version
$$
R\co \mathcal U/BF_{h\mathcal V}\to \mathcal V\mathcal U/BF 
$$
of the $\I$-space lifting functor. Here the homotopy colimit $BF_{h\mathcal V}$ denotes the realization of the simplicial space 
$$
[k]\mapsto \coprod_{n_0,\dots,n_k}\mathcal V(\mathbb R^{n_1},\mathbb R^{n_0})\times
\dots\times \mathcal V(\mathbb R^{n_{k}},\mathbb R^{n_{k-1}})\times BF(n_k).
$$
The statement in Lemma \ref{Blemma} remains true with $\mathcal V$ instead of $\mathcal I$, and we conclude from this that the canonical map $BF_{h\mathcal N}\to BF_{h\mathcal V}$ is a weak homotopy equivalence. The definition of the $\mathcal V$-space lifting functor is then completely analogous to the definition of the $\I$-space lifting functor in Section 
\ref{Rliftingsection}: Let $\overline{BF}$ be the $\mathcal V$-space defined by the homotopy Kan extension along the identity functor on $\mathcal V$. Given a map 
$f\co X\to BF_{h\mathcal V}$, we define $R_f(X)$ to be the homotopy pullback of the diagram of $\mathcal V$-spaces
$$
X\xr{f}BF_{h\mathcal V}\xl{\pi} \overline{BF},
$$
and we define $R(f)$ to be the composition
$$
R(f)\co R_f(X)\to\overline{BF}\to BF.
$$ 
The Barratt-Eccles operad acts on $BF_{h\mathcal V}$ and the results on preservation of operad actions from Section \ref{Thomoperadsection} carry over to this setting.
\appendix

\section{Homotopy colimits}\label{hocolimsection}
We here collect the facts about homotopy colimits needed in the paper. We shall adapt the definitions of Bousfield and Kan \cite{BK}, except that we work with topological spaces instead of simplicial sets. Thus, given a small category $\mathcal A$ and an $\mathcal A$-diagram $X\co \mathcal A\to \mathcal U$, the homotopy colimit $\hocolim_{\mathcal C}X$ is defined to be the realization of the simplicial space
\begin{equation}\label{hocolimdefinition}
[k]\mapsto \coprod_{a_0\leftarrow\dots\leftarrow a_k}X(a_k),
\end{equation}
where the coproduct is over the $k$-simplices of the nerve $N_{\bullet}\mathcal C$. It is sometimes convenient to view this as the classifying space of the topological category $\mathcal A(X)$ whose space of objects is the disjoint union of the spaces $X(a)$ where $a$ runs through the objects of $\mathcal A$. A morphism $(a,x)\to(a',x')$ in $\mathcal A(X)$ is specified by a morphism $\alpha\co a\to a'$ in $\mathcal A$ such that $\alpha_*x=x'$.
If $X$ is a based $\mathcal A$-diagram, that is, a functor $X\co \mathcal A\to \mathcal T$, then the inclusion of the base points gives a map $B\mathcal A\to B\mathcal A(X)$ and we define the based homotopy colimit to be the quotient space. Equivalently, this is the realization of the simplicial space obtained by replacing the disjoint union in (\ref{hocolimdefinition}) by the wedge product. 

\subsection{Homotopy Kan extensions}\label{homotopyKansection}
Let $\phi\co\mathcal A\to\mathcal B$ be a functor between small categories. Given an $\mathcal A$-diagram $X$, the (left) homotopy Kan extension $\phi^h_*X$ is the $\mathcal B$-diagram defined by
$$
\phi^h_*X(b)=\hocolim_{(\phi\downarrow b)}X\circ \pi_b.
$$
The homotopy colimit is over the category $(\phi\downarrow b)$ whose objects are pairs 
$(a,\beta)$ in which $a$ is an object of $\mathcal A$ and $\beta\co\phi(a)\to b$ is a morphism in $\mathcal B$. A morphism $(a,\beta)\to(a',\beta')$ is given by a morphism $\alpha\co a\to a'$ in $\mathcal A$ such that 
$\beta=\beta'\circ \phi(\alpha)$. The functor $\pi_b\co (\phi\downarrow b)\to \mathcal A$ is defined by $(a,\beta)\mapsto a$.  We recall that the categorical Kan extension $\phi_*X$ is defined using the categorical colimit instead of the homotopy colimit, see \cite{Mac}. If $\mathcal B$ is the terminal category $*$ and $p\co\mathcal A\to *$ the projection, then $p_*X$ and $p^h_*X$ are respectively the colimit and the homotopy colimit of the $\mathcal A$-diagram $X$. Notice, that the functors $\pi_b$ define a map of $\mathcal B$-diagrams from $\phi_*^hX$ to the constant $\mathcal B$-diagram $\hocolim_{\mathcal A}X$. A proof of the following well-known lemma can be found in \cite{Sch2}. 

\begin{lemma}\label{homotopyKanequivalence}
The induced map
$$
\pi\co\hocolim_{\mathcal B}\phi^h_*X\xr{\sim}\hocolim_{\mathcal A}X.
$$
is a weak homotopy equivalence. \qed 
\end{lemma}

This lemma may be viewed as a statement about the composition of two derived functors. Given an additional  functor $\psi\co\mathcal B\to\mathcal C$, one can more generally show that there is a natural equivalence of functors $\psi^h_*\phi^h_*\xr{\sim}(\psi\phi)^h_*$. In the lemma below, we shall consider the case where $X$ has the form $\phi^*Y$ for a $\mathcal B$-diagram $Y$, and we shall relate $\pi$ to the map of homotopy colimits induced by the natural transformation of $\mathcal B$-diagrams
$$
t\co \phi_*^h\phi^*Y\to \phi_*\phi^* Y\to Y, 
$$ 
where the first arrow is the canonical projection from the homotopy colimit to the colimit and the second arrow is given by the universal property of the categorical Kan extension.
\begin{lemma}\label{homotopyKanlemma}
Given a $\mathcal B$-diagram $Y$, the diagram
\[
\xymatrix{
\displaystyle\hocolim_{\mathcal B}\phi^h_*\phi^*Y\ar[rr]^{\pi}\ar[dr]_t & & 
\displaystyle\hocolim_{\mathcal A}\phi^*Y\ar[dl]^{\phi}\\
& \displaystyle\hocolim_{\mathcal B}Y  &
}
\]
is homotopy commutative by a canonical choice of a natural homotopy.
\end{lemma}
\begin{proof}
We may view the homotopy colimit of the $\mathcal B$-diagram $\phi^h_*\phi^*Y$ as the realization of the bisimplicial space
$$
([i],[j])\mapsto \coprod_
{\left\{\substack{b_0\la\dots\la b_i\la\phi(a_0)\\a_0\la\dots \la a_j}\right\}}
\phi^*Y(a_j),
$$
and it is well-known that this is homeomorphic to the realization of the diagonal simplicial space. Restricting to this simplicial space, the two maps in the diagram are induced by the simplicial maps that map a simplex
$$
(b_0\la\dots\la b_i\xl{\gamma}\phi(a_0),\ a_0\xl{\alpha_1}\dots\xl{\alpha_i} a_i,\  y)
$$
with $y$ in $\phi^*Y(a_i)$, to
$$
(b_0\la\dots\la b_i,\ \gamma_*\phi(\alpha_1\dots\alpha_i)_*y), \quad\text{respectively}\quad (\phi(a_0)\xl{\phi(\alpha_0)}\dots\xl{\phi(\alpha_i)} \phi(a_i),\ y).
$$
The required homotopy between the topological realizations of these maps is then defined by
$$
[(b_0\la\dots\la b_i\la\phi(a_0)\la\dots\la\phi(a_i),\ y);(su,(1-s)u)],
$$
for $u\in\Delta^i$ and $s\in I$. Here $I$ denotes the unit interval and
$$
\Delta^i=\{(u_0,\dots,u_i)\in I^{i+1}\co u_0+\dots+ u_i=1\}
$$
is the standard $i$-simplex. 
\end{proof}

The following lemma is needed to ensure that the functor $R$ defined in Section 
\ref{Rliftingsection} takes values in the subcategory of level-wise $T$-good objects in 
$\I\mathcal U/BF$.    
\begin{lemma}\label{hocolimquasifibration}
Let $\mathcal A$ be a small category and let $f_a\co X_a\to BF(n)$ be an $\mathcal A$-diagram in $\mathcal U/BF(n)$. If each $f_a$ classifies a well-based quasifibration, then the induced map
$$
f\co\hocolim_{\mathcal A}X_a\to BF(n)
$$ 
also classifies a well-based quasifibration.
\end{lemma}
\begin{proof}
Let $W_a=f_a^*V(n)$, and notice that $f^*V(n)$ is homeomorphic to 
$\hocolim_{\mathcal A}W_a$ since topological realization preserves pullback diagrams.
It follows that the pullback of $V(n)\to BF(n)$ along $f$ is homeomorphic to the realization of the simplicial map
$$
\coprod_{a_0\leftarrow\dots\leftarrow a_k}W_{a_k}\to \coprod_{a_0\leftarrow\dots\leftarrow a_k}
X_{a_k}.
$$ 
These are good simplicial spaces in the sense of \cite{Se}, Appendix A, 
and the map is a degree-wise quasifibration by assumption. 
The result then follows from standard results on realization of simplicial quasifibrations and simplicial cofibrations, see e.g. \cite{Se}, Proposition 1.6 and \cite{Le}.  
\end{proof}


\begin{thebibliography}{99}
\bibitem{AB}
R. D. Arthan and S. R.  Bullett, \emph{The homology of $MO(1)\sp{\wedge \infty }$ and $MU(1)\sp{\wedge \infty }$}, J. Pure Appl. Algebra 26 (1982), no. 3, 229--234.
\bibitem{BE}
M. G. Barratt and P. J. Eccles, \emph{$\Gamma^+$-structures-I: A free group functor for stable homotopy theory}, Topology 13 (1974), 25--45. 
\bibitem{Ber}
C. Berger, \emph{Combinatorial models for real configuration spaces and $E_n$-operads}, Contemp. Math., vol. 202, 37--52, Amer. Math. Soc., Providence, RI, 1997. 
\bibitem{BCS}
A. Blumberg, R. Cohen, and C. Schlichtkrull, \emph{Topological Hochschild homology of Thom spectra and the free loop space}, Preprint, 2008. 
\bibitem{Bi}
J. S. Birman, \emph{Braids, links, and mapping class groups}, Annals of Mathematics Studies, No 82, Princeton University Press, Princeton, N.J., 1975. 
\bibitem{Bo}
F. Borceux, \emph{Handbook of categorical algebra 2: Categories and structures}, Encyclopedia of Mathematics and its Applications 51,  Cambridge University Press, Cambridge, 1994.  
\bibitem{BK}
A. K. Bousfield and D. M. Kan, \emph{Homotopy limits, completions and
  localizations}, Springer Lecture Notes in Math. 304, Springer Verlag, 1972.
\bibitem{Bok}
M. B\"okstedt, \emph{Topological Hochschild homology}, Preprint 1985, Bielefeld University.  
\bibitem{Bu}
S. R. Bullett, \emph{Permutations and braids in cobordism theory}, Proc. London Math. Soc. (3) 38 (1979), no. 3, 517--531. 
\bibitem{Co}
F. R. Cohen, \emph{Braid orientations and bundles with flat connection}, Inventiones math. 46, (1978), 99--110. 
\bibitem{CLM}
F. R. Cohen, T. Lada, and J. P. May, \emph{Homology of iterated loop spaces}, Springer Lecture Notes in Math. 533, Springer Verlag, 1976.
\bibitem{CV}
F. R. Cohen and V. V. Vershinin, \emph{Thom spectra which are wedges of Eilenberg-Mac Lane spectra}, Stable and unstable homotopy (Toronto, ON, 1996), 43--65, 
Fields Inst. Commun., 19, Amer. Math. Soc., Providence, RI, 1998.
\bibitem{DS}
W. G. Dwyer and J. Spalinski, \emph{Homotopy theories and model categories}, Handbook of algebraic topology, North-Holland, Amsterdam, 1995, pp. 73--126.
\bibitem{G}
T. Goodwillie, \emph{Calculus II: Analytic functors}, K-theory 5,
(1992), no. 4, 295--332.
\bibitem{Hi}
P. S. Hirschhorn, \emph{Model categories and their localizations}, Mathematical Surveys and Monographs 99, American Mathematical Society, Providence, R.I, 2003.
\bibitem{HSS}
M. Hovey, B. Shipley and J. Smith, \emph{Symmetric spectra},
J. Amer. Math. Soc. 13, (2000), no. 1, 149--208.
\bibitem{HV}
J. Hollender and R. M. Vogt, \emph{Modules of topological spaces, applications to homotopy limits and $E_{\infty}$ structures}, Arch. Math. 59 (1992), no. 2, pp. 115--129.
\bibitem{Ja}
I. M. James, \emph{Reduced product spaces}, Ann. of Math. (2) 62, (1955), 170--197.
\bibitem{Le}
L. G. Lewis, \emph{When is the natural map $X\to\Omega\Sigma X$ a
  cofibration?}, Trans. Amer. Math. Soc. 273 (1982), no. 1,  147--155.  
\bibitem{LMS}
L. G. Lewis, J. P. May and M. Steinberger, \emph{Equivariant stable
  homotopy theory}, Springer Lecture Notes in Math. 1213, Springer
Verlag, 1986.  
\bibitem{Li}
J. Lillig, \emph{A union theorem for cofibrations}, Arch. Math 24
(1973), pp. 410--415.
\bibitem{Mac}
S. Mac Lane, \emph{Categories for the working mathematician}. Second Edition, Graduate Texts in Mathematics, Springer Verlag, New York, 1998.
Springer-Verlag, 1971.
\bibitem{M}
I. Madsen, \emph{Algebraic K-theory and traces}, from: ``Current Developments in Mathematics, 1995, (Cambridge, MA), International Press, pp. 191--321 (1996) . 
\bibitem{Mah1}
M. Mahowald, \emph{A new infinite family in ${}_2\pi^s_*$}, Topology 16 (1977), 249--256. 
\bibitem{Mah}
M. Mahowald, \emph{Ring spectra which are Thom spectra}, Duke Math. J. 46 (1979), 549--559.
\bibitem{MR}
M. Mahowald and N. Ray, \emph{A note on the Thom isomorphism}, Proc. Amer. Math. Soc. 82 (1981), no. 2, 307--308.
\bibitem{MMSS}
M. A. Mandell, J. P. May, S. Schwede and B. Shipley, \emph{Model
  categories of diagram spectra}, Proc. London Math. Soc. (3), 82
(2001), no. 2, 441--512.
\bibitem{Ma}
J. P. May, \emph{The geometry of iterated loop spaces}, Springer
Lecture Notes in Math. 271, Springer Verlag, 1972.
\bibitem{Ma1}
J. P. May, \emph{$E\sb{\infty }$ spaces, group completions, and permutative categories}, New developments in topology (Proc. Sympos. Algebraic Topology, Oxford, 1972), pp. 61--93. London Math. Soc. Lecture Note Ser., No. 11, Cambridge Univ. Press, London, 1974. 
\bibitem{Ma2}
J. P. May, \emph{Classifying spaces and fibrations},
Mem. Amer. Math. Soc. 1 (1975), 1, no. 155. 
\bibitem{Ma3}
J. P. May, \emph{$E_{\infty}$ ring spaces and $E_{\infty}$ ring spectra}, Springer Lecture Notes in Math. 577, Springer Verlag, 1977. 
\bibitem{May4}
J. P. May and J. Sigurdsson, \emph{Parametrized homotopy theory}, Mathematical Surveys and Monographs, 132, American Mathematical Society, Providence, RI, 2006. 
\bibitem{SS}
S. Sagave and C. Schlichtkrull, \emph{Diagram spaces and symmetric spectra}, in preparation.
\bibitem{Sch2}
C. Schlichtkrull, \emph{The homotopy infinite symmetric product represents stable homotopy}, 
Algebr. Geom. Topol. 7 (2007), 1963--1977.
\bibitem{Schw}
S. Schwede, \emph{Symmetric spectra}, in preparation, 
\verb+http://www.math.uni-bonn.de/people/schwede/SymSpec.pdf+.
\bibitem{Se}
G. Segal, \emph{Categories and cohomology theories}, Topology 13
(1974),  293--312.
\bibitem{Sh}
B. Shipley, \emph{Symmetric spectra and topological Hochschild
  homology}, $K$-theory 19, (2000), no. 2, 155--183.
\bibitem{Sm}
J. H. Smith, \emph{Simplicial group models for $\Omega^nS^nX$}, Israel J. Math. 66 (1989), no. 1--3, 330--350.
\bibitem{St}
N. Steenrod, \emph {A convenient category of topological spaces},
  Michigan Math. J. 14 (1967), 133--152.
\bibitem{Str}
A. Str\o m, \emph{Note on cofibrations II}, Math. Scand. 22 (1968), 130--142.
\bibitem{Str2}
A. Str\o m, \emph{The homotopy category is a homotopy category},
Arch. Math. 23 (1972), 435--441.
\bibitem{W}
F. Waldhausen, \emph{Algebraic K-theory of spaces}, Lecture Notes in Math. 1126, 318--419, Springer-Verlag, New York, 1985. 

\end{thebibliography}
\end{document}